\documentclass[a4paper,12pt]{amsart}  

\usepackage{amsmath,amsthm,amssymb}
\usepackage[all]{xy}
\usepackage{setspace}
\allowdisplaybreaks
\usepackage{color}

\newtheorem{theorem}{Theorem}[section]
\newtheorem{proposition}[theorem]{Proposition}
\newtheorem{definition}[theorem]{Definition}
\newtheorem{lemma}[theorem]{Lemma}
\newtheorem{corollary}[theorem]{Corollary}
\newtheorem{conjecture}[theorem]{Conjecture}

\theoremstyle{remark}
\newtheorem{remark}[theorem]{Remark}

\numberwithin{equation}{section}

\DeclareFontEncoding{OT2}{}{}
\DeclareFontSubstitution{OT2}{cmr}{m}{n}
\DeclareFontFamily{OT2}{cmr}{\hyphenchar\font45 }
\DeclareFontShape{OT2}{cmr}{m}{n}{<5><6><7><8><9>gen*wncyr<10><10.95><12><14.4><17.28><20.74><24.88>wncyr10}{}
\DeclareFontShape{OT2}{cmr}{b}{n}{<5><6><7><8><9>gen*wncyb<10><10.95><12><14.4><17.28><20.74><24.88>wncyb10}{}
\DeclareMathAlphabet{\mathcyr}{OT2}{cmr}{m}{n}
\DeclareMathAlphabet{\mathcyb}{OT2}{cmr}{b}{n}
\SetMathAlphabet{\mathcyr}{bold}{OT2}{cmr}{b}{n}

\usepackage{shuffle}

\newcommand{\R}{\mathbb R}
\newcommand{\Z}{{\mathbb Z}}

\newcommand{\C}{{\mathbb C}}
\newcommand{\Q}{{\mathbb Q}}

\newcommand{\SL}{\rm SL}

\newcommand{\odd}{\mathbb O}
\newcommand{\even}{\mathbb E}
\newcommand{\cusp}{\mathbb S}

\newcommand{\Zsp}{\mathcal{H}}
\newcommand{\Asp}{\mathcal{A}}

\newcommand{\depth}{{\mathfrak D}}
\newcommand{\m}{{\mathfrak m}}

\newcommand{\rank}{{\rm rank}\,}
\newcommand{\V}{{\mathbf{V}}}

\title[Multiple zeta values and period polynomials]{Relationships between multiple zeta values of depths 2 and 3 and period polynomials}
\author[D.~Ma, K.~Tasaka]{Ding Ma, Koji Tasaka}
\address[Ding Ma]{Department of Mathematics
Duke university, US}
\email{martin@math.duke.edu}
\address[Koji Tasaka]{Department of Information Science and Technology, Aichi Prefectural University}
\email{tasaka@ist.aichi-pu.ac.jp}
\date{}

\begin{document}

\maketitle

\begin{abstract}
Some combinatorial aspects of relations between multiple zeta values of depths 2 and 3 and period polynomials are discussed. 
\end{abstract}


\section{Introduction}

The multiple zeta value is defined by 
\begin{equation*}
\zeta(k_1,\ldots,k_r) = \sum_{0<m_1<\cdots<m_r} \frac{1}{m_1^{k_1}\cdots m_r^{k_r}},
\end{equation*}
for integers $k_1,\ldots,k_{r-1}\ge1$ and $k_r\ge2$. 
We call $k_1+\cdots+k_r$ the weight and $r$ the depth.
Multiple zeta values have various connections with modular forms on $\SL_2(\Z)$ (or their period polynomials), which can be found in many articles \cite{BS,GKZ,KT,M1,M2,M3,M4,Ma,Z,Z1,Z3} for depth 2 and \cite{BK,B,B2,CGS,EL,G2,T} for depth $>2$.
In the present paper, we examine explicit connections between multiple zeta values of depths 2 and 3 and period polynomials.

Our results on depth 2 and on depth 3 lie in a slightly different context. 
For depth 2 case, motivated by the work of Gangl, Kaneko and Zagier \cite{GKZ}, we give a direct connection between even period polynomials of cusp forms and linear relations among $\zeta(odd,odd)$'s (see Theorem \ref{main1}).
A similar connection for the case $\zeta(odd,even)$, which was developed by the first author in \cite{M1,M2}, will be also described in our setting (see Theorem \ref{thm:Ding-odd}).
In depth 3, we give some indirect connections between period polynomials and linear relations among almost totally odd triple zeta values $\zeta(odd,odd,even)$, $\zeta(odd,even,odd)$ and $\zeta(even,odd,odd)$ (see Theorems \ref{thm:rest_even_C^j}, \ref{thm:B^3} and \ref{thm:B^3+E^3}).
These results can be viewed as generalizations of results given for the case $\zeta(odd,odd)$ by Baumard and Schneps \cite{BS} and for the case $\zeta(odd,even)$ by Zagier \cite{Z3}, and lead to upper bounds of the dimension of the $\Q$-vector space spanned by almost totally odd triple zeta values.
An explicit formula for the parity theorem (Theorem \ref{9_1}) is also derived from our discussions. 
Our proofs are based on the theory of motivic multiple zeta values developed by Brown (see e.g. \cite{B1,B3}), which will be summarized in Section 2.
In particular, Brown's method \cite[\S3]{B1} for proving linear relations among motivic multiple zeta values modulo single motivic zeta values will play a crucial role.

\subsection*{Acknowledgement}
The authors are grateful to Francis Brown for initial advice and useful comments.
The second author wishes to express his thanks to Herbert Gangl for drawing his attention to Yamamoto's $\frac{1}{2}$-interpolated multiple zeta values .
The second author also thanks Max Planck Institute for Mathematics, where the paper was written, for the invitation and hospitality.
This work is partially supported by Japan Society for the Promotion of Science, Grant-in-Aid for JSPS Fellows (No. 16H07115).

\section{Preliminaries}

\subsection{Motivic multiple zeta values}

We follow the notation of \cite{B1}.
The definition of the motivic multiple zeta value $\zeta^{\m}(n_1,\ldots,n_r)$ we use is referred to \cite[Definition 2.1]{B1}, where $\zeta^\m(2)$ is not treated to be zero.
A more elaborate definition can be found in \cite[\S2.2]{B3}, where the motivic multiple zeta value is defined as a motivic period of the Tannakian category of the mixed Tate motives over $\Z$.

Let $\mathcal{H}$ be the $\Q$-vector space spanned by all motivic multiple zeta values.
As usual, we call $n_1+\cdots+n_r$ the weight and $r$ the depth for $\zeta^{\m}(n_1,\ldots,n_r)$.
We regard $1\in \Q$ as the unique motivic multiple zeta value of weight 0 and depth 0.
Let $\Zsp_N$ denote the $\Q$-vector space spanned by all motivic multiple zeta values of weight $N$.
The space $\mathcal{H}$ naturally has the structure of a graded $\Q$-algebra
\[ \mathcal{H} = \bigoplus_{N\ge0} \mathcal{H}_N\]
with the shuffle product $\shuffle$.
There is the period map (see \cite[Eq.~(2.11)]{B1})
\begin{equation}\label{eq:per} 
{\rm per}: \mathcal{H} \longrightarrow \R
\end{equation}
that send $\zeta^\m(n_1,\ldots,n_r)$ to $\zeta^\shuffle (n_1,\ldots,n_r)$ the shuffle regularized multiple zeta value (see e.g. \cite[\S2]{IKZ} for the definition of $\zeta^\shuffle $).
We note that $\zeta^\m(1)=0$.

Let $\mathcal{A}$ be the quotient algebra $\mathcal{H}/\zeta^m(2)\mathcal{H}$.
From \cite[Theorem 2.4]{B1}, we see that the space $\mathcal{H}$ forms a graded $\mathcal{A}$-comodule over $\Q$ with the coaction $\Delta : \mathcal{H}\rightarrow \mathcal{A}\otimes \mathcal{H}$.
The coaction $\Delta$ is computed from Goncharov's coproduct formula for motivic iterated integrals with the factors interchanged (referred to \cite[Eq.~(2.18)]{B1}, and \cite[Eq.~(27)]{G} for the original one). 

As a partial answer to the Hoffman conjecture \cite[Conjecture C]{H}, it was shown by Brown \cite[Theorem 1.1]{B1} that the set $\{\zeta^\m(n_1,\ldots,n_r) \mid r\ge0, n_1,\ldots,n_r\in\{2,3\}\}$ forms a basis of the $\Q$-vector space $\mathcal{H}$.
There are several important consequences of Brown's theorem \cite[Theorem 1.1]{B1}, of which we use the following version.
\begin{theorem}\label{thm:phi}
There is a non-canonical isomorphism 
\[ \phi :  \Zsp\longrightarrow \mathcal{U}:=\Q\langle f_{2i+1}\mid i\ge1 \rangle\otimes_\Q \Q[f_2]\]
as algebra-comodules, sending $\zeta^\m(k)$ to $f_k$ for $k\ge2$, where we put
\[ f_{2k}:=\frac{\zeta^\m(2k)}{\zeta^\m(2)^k}f_2^k \ (k\ge1).\]
\end{theorem}

The structure on $\mathcal{U}$ is as follows.
The noncommutative polynomial ring $\Q\langle f_{2i+1}\mid i\ge1 \rangle$ freely generated by symbols $f_{2i+1}$ in weight $2i+1$ is viewed as the universal enveloping algebra of the graded Lie algebra of $U_{dR}$ the prounipotent part of the motivic Galois group over $\Q$, i.e. a Hopf algebra with the product given by the shuffle product $\shuffle$ and the coproduct given by the deconcatenation $\Delta^\bullet$ referred to \cite[Eq.~(2.20)]{B1}.
With a commutative symbol $f_2$ of weight $2$, the vector space $\mathcal{U}=\Q\langle f_{2i+1}\mid i\ge1 \rangle\otimes_\Q \Q[f_2]$ forms a graded vector space $\mathcal{U}=\bigoplus_{N\ge0} \mathcal{U}_N$, where $\mathcal{U}_N$ denotes the $\Q$-vector subspace of $\mathcal{U}$ spanned by all words $f_{2a_1+1}\cdots f_{2a_r+1}f_2^k \ (=f_{2a_1+1}\cdots f_{2a_r+1}\otimes f_2^k)$ with $2a_1+1+\cdots +2a_r+1+2k=N$ ($a_1,\ldots,a_r\ge1,k\ge0$). 
The space $\mathcal{U}$ carries a graded $\mathcal{U}'$-comodule structure with the coaction $\Delta^\bullet : \mathcal{U} \rightarrow \mathcal{U}'\otimes \mathcal{U}$ such that $\Delta^\bullet(f_2)=1\otimes f_2$ and $\Delta^\bullet(wf_2^k)=\Delta^\bullet(w)\Delta^\bullet(f_2)^k$ for any $k>0$ and $w\in \mathcal{U}'$.

\

We remark that the map $\phi$ in Theorem \ref{thm:phi} is a composition of the maps Eq.~(2.15) and Eq.~(2.22) in \cite{B1}.
The map $\phi$ is non-canonical, but we make a choice of $\phi$ for depth 1 in Theorem \ref{thm:phi} (the compatibility follows from \cite[Lemma 3.2]{B1}).  
In what follows, we aim at computing an image of the motivic multiple zeta value of depths 2 and 3 under the map $\phi$.

\subsection{The parity theorem}
Since we often use the parity theorem (and its motivic version), we recall its basics. 

The parity theorem states that the multiple zeta value $\zeta(k_1,\ldots,k_r)$ can be written as a $\Q[\zeta(2)]$-linear combination of multiple zeta values of depth at most $r-1$, if $k_1+\cdots+k_r\not\equiv r \mod{2}$.
This theorem can be proved by using the regularized double shuffle relation modulo lower depths \cite[Proposition 17]{IKZ} (see also \cite[Proposition 6.4]{B1}).
Since motivic multiple zeta values satisfy the regularized double shuffle relation (see \cite[Theorem 7.4]{G2}), the parity theorem holds for the motivic multiple zeta value.
For other approaches to the parity theorem, see \cite{Tsumura}.

Explicit formulas for the parity theorem are not known in general. 
A few cases are known: depth 2 case is due to Zagier \cite[Proposition 7]{Z3} and depth 3 case is obtained by Panzer \cite[Eq.~(1.12)]{P}.
The proof of Zagier's formula uses the regularized double shuffle relation in depth 2, so it is lifted to motivic, but we do not know if Panzer's formula is lifted or not.
As an application of our usage of the theory of motivic multiple zeta values, a kind of explicit formula of the parity theorem of motivic triple zeta values will be obtained (see Theorem \ref{9_1}).

\subsection{Key lemma}

In this subsection, we first define infinitesimal coactions $D_m$ and $D_m^\bullet$ and then recall \cite[Lemma 2.4]{B1}.

Denote by $\mathcal{A}_N$ the image of $\mathcal{H}_N$ under the natural projection $\mathcal{H}\rightarrow \mathcal{A}$.
Set $\mathcal{A}_{>0}=\bigoplus_{N\ge1}\mathcal{A}_N$ and $\mathcal{L}=\Asp_{>0}/(\Asp_{>0})^2$.
Let $\pi_m:\Asp_{>0}\rightarrow \mathcal{L}_m$ be the natural projection taking the graded weight $m$ part $\mathcal{L}_m$ of the graded vector space $\mathcal{L}=\bigoplus_{m>0}\mathcal{L}_m$. 
An infinitesimal coaction $D_m$ of $\Delta$ is defined for all odd integer $m\ge3$ by the following composition map:
\[ D_m : \mathcal{H} \overset{\Delta-\epsilon\otimes {\rm id}}{\longrightarrow}  \mathcal{A}_{>0} \otimes  \mathcal{H} \overset{\pi_m\otimes {\rm id}}{\longrightarrow} \mathcal{L}_m \otimes  \mathcal{H},\]
where $\epsilon:\mathcal{A}\rightarrow \mathcal{A}_{>0}$ is counit that kills $\Q=\mathcal{A}_0$.
The infinitesimal coaction $D_m$ can be computed by motivic iterated integrals (see \cite[Eq.~(3.4)]{B1} for the formula).
In what follows, we denote by $\xi_m$ the image of $\zeta^\m(m)$ under the projection:
\begin{align*}
\mathcal{H}_m&\longrightarrow \mathcal{L}_m \\
\zeta^\m(m) &\longmapsto \xi_m.
\end{align*} 

Let $\mathcal{U}'=\Q\langle f_{2i+1}\mid i\ge1 \rangle$. 
It is graded by weight $\mathcal{U}'=\bigoplus_{N\ge0} \mathcal{U}'_N$.
Set $\mathcal{U}_{>0}'= \bigoplus_{N>0} \mathcal{U}'_N$ and $L = \bigoplus_{m>0} L_m = \mathcal{U}_{>0}' \big/ \big(\mathcal{U}_{>0}' \big)^2$.
Denote by $\pi_m' :\mathcal{U}_{>0}'\rightarrow L_m$ the projection.
An infinitesimal coaction $D_m^\bullet$ of the deconcatenation $\Delta^\bullet$ (see \cite[Eq.~(2.25)]{B1}) is defined in the same manner as $D_m$ by 
\[ D^\bullet_{m} : \mathcal{U} \overset{\Delta^\bullet-\epsilon'\otimes {\rm id}}{\longrightarrow} \mathcal{U}_{>0}'\otimes \mathcal{U} \overset{\pi_m'\otimes {\rm id}}{\longrightarrow} L_m\otimes \mathcal{U},\]
where $\epsilon':\mathcal{U}'\rightarrow \mathcal{U}'_{>0}$ is counit.

\begin{lemma}\label{lem:key}
(i) For $m\ge3$ odd the following diagram commutes:
\begin{equation*}
\begin{aligned}
\xymatrix{  
 \Zsp   \ar@{->}[r]^{D_m} \ar@{->}[d]^{\phi} & \mathcal{L}_m \otimes \Zsp \ar@{->}[d]^{\overline{\phi}\otimes \phi} \\ 
 \mathcal{U}   \ar@{->}[r]^{D_m^\bullet}  & L_m \otimes \mathcal{U}  }
 \end{aligned}\, ,
 \end{equation*}
where $\overline{\phi}:\mathcal{L} \rightarrow L$ is the induced homomorphism from $\phi$ that sends $\xi_m$ to $\pi_m'(f_m)$. \\
(ii) {\rm (\cite[Lemma 2.4]{B1})} 
We have
\[ \ker \sum_{\substack{1<m<N\\m:{\rm odd}}} D_m^\bullet\big|_{\mathcal{U}_N} = \Q f_N.\]
\end{lemma}
\begin{proof}
The statement i) follows from the fact that the map $\phi$ is an algebra-comodule homomorphism. 
For the proof of the statement ii), we refer to the reader to \cite[Lemma 2.4]{B1}.
\end{proof}

\subsection{A formula for $D_m$}

The infinitesimal coaction $D_m$ for motivic multiple zeta values is computed in several papers (\cite{B0,B1} are the first).
We describe an explicit formula for $D_m(\zeta^\m(n_1,\ldots,n_r))$ modulo lower depths given by Glanois \cite{Glanois} (a typo was corrected by Ma \cite[Proposition 8.2.1]{Ma}).

Let $\delta{\tbinom{m_1,\ldots,m_r}{n_1,\ldots,n_r}}$ be the Kronecker delta defined by
\[ \delta{\tbinom{m_1,\ldots,m_r}{n_1,\ldots,n_r}}= \begin{cases} 1 & \mbox{if $m_i=n_i$ for all $i\in \{ 1,\ldots,r\}$}\\
0 & \mbox{otherwise} \end{cases} \]
with $\delta(\varnothing)=1$.
We define the integer $b_{n,n'}^m$ for $n,n'\in \Z$ and $m\in\Z_{\ge1}$ by
\begin{equation*}
b_{n,n'}^{m}=(-1)^{n}\binom{m-1}{n-1}+(-1)^{n'-m} \binom{m-1}{n'-1},
\end{equation*}
where $\binom{m}{n}=0$ for each $n<0$ and for the case $m<n$.
It is obvious that for any odd integer $m\ge3$ one has $b_{n,n'}^m+b_{n',n}^m=0$.

\begin{definition}
For $r\ge1$ and $r$-tuples of positive integers $(m_1,\ldots,m_r)$ and $(n_1,\ldots,n_r)$, we define
\begin{equation*}
e{\tbinom{m_1,\ldots,m_r}{n_1,\ldots,n_r}}= \delta{\tbinom{m_1,\ldots,m_r}{n_1,\ldots,n_r}}+ \sum_{i=1}^{r-1} \delta{\tbinom{m_2,\ldots,m_i,m_{i+2},\ldots,m_r}{n_1,\ldots,n_{i-1},n_{i+2},\ldots,n_r}} b_{n_{i},n_{i+1}}^{m_1} \in \Z.
\end{equation*}
\end{definition}

For example, we have $e\tbinom{m_1}{n_1}=\delta\tbinom{m_1}{n_1}$ for $r=1$,  $e\tbinom{m_1,m_2}{n_1,n_2}=\delta\tbinom{m_1,m_2}{n_1,n_2}+b_{n_1,n_2}^{m_1}$ for $r=2$ and  
\begin{equation}\label{eq:def_e3}
 e\tbinom{m_1,m_2,m_3}{n_1,n_2,n_3}= \delta\tbinom{m_1,m_2,m_3}{n_1,n_2,n_3} + \delta\tbinom{m_3}{n_3}b_{n_1,n_2}^{m_1} + \delta\tbinom{m_2}{n_1}b_{n_2,n_3}^{m_1}\quad \mbox{for} \ r=3.
\end{equation}
Denote by $\depth_r\Zsp$ the $\Q$-vector space spanned by all motivic multiple zeta values of depth $\le r$:
\[ \depth_r\Zsp = \langle \zeta^\m (n_1,\ldots,n_s)\mid n_1,\ldots,n_s\ge1,0\le s \le r \rangle_\Q.\]
The following proposition is derived from \cite[Proposition 8.2.1]{Ma} for level 1 (see also \cite[Lemma 2.8]{Glanois}).

\begin{proposition}\label{prop:D_m}
For any integers $n_1,\ldots,n_r\ge1$ with $N=n_1+\cdots+n_r$ and $m\ge3$ odd, the element
\[
D_m \big( \zeta^\m(n_1,\ldots,n_r) \big) - \sum_{\substack{m_1+\cdots+m_r=N\\ m_1,\ldots,m_r\ge1}} \delta\tbinom{m_1}{m}e{\tbinom{m_1,\ldots,m_r}{n_1,\ldots,n_r}} \xi_{m_1} \otimes \zeta^\m(m_2,\ldots,m_r)
\]
lies in $\mathcal{L}_{m} \otimes \depth_{r-2} \Zsp_{N-m}$, where $\depth_r\Zsp_N=\depth_r\Zsp\cap \Zsp_N$, where $\xi_m$ is an image of $\zeta^\m(m)$ in the space $\mathcal{L}$.
\end{proposition}

\subsection{A canonical part of $\phi$}

Using Proposition \ref{prop:D_m}, one can compute leading terms of $\phi(\zeta^\m(n_1,\ldots,n_r))$ for $r\le3$.

\begin{definition}\label{def:c}
Let $N,m_1,m_2,m_3,n_1,n_2,n_3$ be positive integers such that $N=m_1+m_2+m_3=n_1+n_2+n_3$, $m_1,m_2\ge3$ odd and $m_3\ge2$.
We define the integer $c\tbinom{m_1,m_2,m_3}{n_1,n_2,n_3}$ by
 \begin{equation*}
c\tbinom{m_1,m_2,m_3}{n_1,n_2,n_3} = \sum_{\substack{k_1+k_2+k_3=N\\ k_1,k_2,k_3\ge1}}  \delta\tbinom{m_1}{k_1}e\tbinom{m_2,m_3}{k_2,k_3} e\tbinom{k_1,k_2,k_3}{n_1,n_2,n_3} .
\end{equation*}
\end{definition}

\begin{proposition}\label{prop:formula_phi}
i) For positive integers $n_1,n_2$, we have
\begin{equation}\label{eq:dep2_phi} 
\phi(\zeta^\m(n_1,n_2)) - \sum_{\substack{m_1+m_2=n_1+n_2\\m_1\ge3;{\rm odd}\\m_2\ge2}} e\tbinom{m_1,m_2}{n_1,n_2} f_{m_1}f_{m_2} \in\Q f_{n_1+n_2}.
\end{equation}
ii) For positive integers $n_1,n_2,n_3$, we have
\begin{equation}\label{eq:dep3_phi}  
\phi(\zeta^\m(n_1,n_2,n_3) ) -\sum_{\substack{m_1+m_2+m_3=N\\m_1,m_2\ge3:{\rm odd}\\m_3\ge2}} c\tbinom{m_1,m_2,m_3}{n_1,n_2,n_3} f_{m_1}f_{m_2}f_{m_3} \in \mathcal{U}_{N,2},
\end{equation}
where $N=n_1+n_2+n_3$ and $ \mathcal{U}_{N,2}$ is the $\Q$-vector space spanned by $f_{2n+1}f_{N-2n-1} \ (1\le n< (N-1)/2)$ and $f_N$.
\end{proposition}

\begin{proof}
We prove i). 
It follows from \cite[Eq.~(3.4)]{B1} that the depth is preserved by $D_m$.
Thus, Proposition \ref{prop:D_m} for the case $r=2$ implies that
\[ D_m (\zeta^\m(n_1,n_2)) = e\tbinom{m,N-m}{n_1,n_2} \xi_{m}\otimes \zeta^\m (N-m)\]
holds for $m\ge3$ odd with $N=n_1+n_2$.
Using the commutative diagram in Lemma \ref{lem:key} i), one computes
\begin{align*}
\sum_{\substack{1<m<N\\m:{\rm odd}}} D_m^\bullet \circ \phi (\zeta^\m(n_1,n_2)) &= \sum_{\substack{1<m<N\\m:{\rm odd}}} (\overline{\phi} \otimes \phi)\circ D_m  (\zeta^\m(n_1,n_2)) \\
&= \sum_{\substack{1<m<N\\m:{\rm odd}}} e\tbinom{m,N-m}{n_1,n_2} \overline{f}_m \otimes f_{N-m},
\end{align*}
where we set $\pi'_m(f_m)=\overline{f}_m$ and put $f_1=0$ when $m=N-1$.
On the other hand, a straightforward calculation of $D_m^\bullet$ gives
\begin{align*}
& \sum_{\substack{1<m<N\\m:{\rm odd}}} D_{m}^\bullet \big(\sum_{\substack{m_1+m_2=N\\m_1\ge3:{\rm odd}\\m_2\ge2}}e\tbinom{m_1,m_2}{n_1,n_2} f_{m_1} f_{m_2} \big) = \sum_{\substack{m_1+m_2=N\\m_1\ge3:{\rm odd}\\m_2\ge2}}e\tbinom{m_1,m_2}{n_1,n_2} \overline{f}_{m_1} \otimes f_{m_2}.
\end{align*}
Thus, one gets 
\[ \phi (\zeta^\m(n_1,n_2)) - \sum_{\substack{m_1+m_2=N\\m_1\ge3:{\rm odd}\\m_2\ge2}}e\tbinom{m_1,m_2}{n_1,n_2} f_{m_1} f_{m_2} \in \ker \sum_{\substack{1<m<N\\m:{\rm odd}}} D_m^\bullet,\]
and by Lemma \ref{lem:key} ii) we have \eqref{eq:dep2_phi}.

Let us turn to the proof of ii). 
From Proposition \ref{prop:D_m} for the case $r=3$, we have
\begin{equation}\label{eq:formula_Dm_dep3}
\begin{aligned}
D_m(\zeta^\m(n_1,n_2,n_3)) - \sum_{\substack{k_1+k_2+k_3Nk\\k_1,k_2,k_3\ge1}} \delta\tbinom{m}{k_1} e\tbinom{k_1,k_2,k_3}{n_1,n_2,n_3} \xi_{k_1}\otimes \zeta^\m (k_2,k_3)\in \Q \mathcal{L}_m \otimes \Q \zeta^\m (N-m).
\end{aligned}
\end{equation}
Since the depth is preserved by $D_m$, the left-hand side of \eqref{eq:formula_Dm_dep3} actually lies in $\depth_2 \mathcal{L}_m \otimes \Q \zeta^\m (k-m)$, where $\depth_r \mathcal{L}_m$ is a natural image of $\depth_r \mathcal{H}_m$.
Since $m\ge3$ odd, it follows from the parity theorem that $\depth_2 \mathcal{L}_m=\Q\xi_m$, and hence, the left-hand side of \eqref{eq:formula_Dm_dep3} is a rational multiple of $\xi_m\otimes \zeta^\m (N-m)$.
Therefore, using the commutative diagram in Lemma \ref{lem:key} i), we see that there is a rational number $c_m$ such that
\begin{align*}
&D_m^\bullet \circ \phi (\zeta^\m(n_1,n_2,n_3))
=(\overline{\phi} \otimes \phi)\circ D_m   (\zeta^\m(n_1,n_2,n_3)) \\
&= \sum_{\substack{k_1+k_2+k_3=N\\k_1,k_2,k_3\ge1}} e\tbinom{k_1,k_2,k_3}{n_1,n_2,n_3} \pi_m'(f_{k_1})\otimes \phi(\zeta^\m(k_2,k_3))+ c_m \overline{f}_m \otimes f_{N-m}.
\end{align*}
Summing these up and making use of \eqref{eq:dep2_phi} and Definition \ref{def:c}, we have
\begin{align*}  
& \sum_{\substack{1<m<k\\m:{\rm odd}}} D_m^\bullet \circ \phi (\zeta^\m(n_1,n_2,n_3)) \\
& = \sum_{\substack{m_1+m_2+m_3=k\\m_1,m_2\ge3:{\rm odd}\\m_3\ge2}}  c\tbinom{m_1,m_2,m_3}{n_1,n_2,n_3}\overline{f}_{m_1}\otimes f_{m_2}f_{m_3} + \sum_{\substack{1<m<k\\m:{\rm odd}}} c_m' \overline{f}_m \otimes f_{k-m}
\end{align*}
with some $c_m'\in\Q$.
On the other hand, by definition one has
\begin{align*}
&\sum_{\substack{1<m<k\\m:{\rm odd}}} D_m^\bullet \Big( \sum_{\substack{m_1+m_2+m_3=k\\m_1,m_2\ge3:{\rm odd}\\m_3\ge2}}  c\tbinom{m_1,m_2,m_3}{n_1,n_2,n_3}f_{m_1}f_{m_2}f_{m_3}  \Big)=  \sum_{\substack{m_1+m_2+m_3=k\\m_1,m_2\ge3:{\rm odd}\\m_3\ge2}} c\tbinom{m_1,m_2,m_3}{n_1,n_2,n_3}\overline{f}_{m_1}\otimes f_{m_2}f_{m_3}.
\end{align*}
Hence the element
\begin{align*}
& \phi (\zeta^\m(n_1,n_2,n_3)) - \sum_{\substack{m_1+m_2+m_3=k\\m_1,m_2\ge3:{\rm odd}\\m_3\ge2}} c\tbinom{m_1,m_2,m_3}{n_1,n_2,n_3}f_{m_1}f_{m_2}f_{m_3} - \sum_{\substack{1<m<k\\m:{\rm odd}}} c_m'f_m f_{k-m}
\end{align*}
lies in $\ker \sum_{\substack{1<m<k\\m:{\rm odd}}} D_m^\bullet$. 
The formula \eqref{eq:dep3_phi} follows from Lemma \ref{lem:key} ii).
\end{proof}


From the proof, we see that the integers $e\tbinom{m_1,m_2}{n_1,n_2}$ and $c\tbinom{m_1,m_2,m_3}{n_1,n_2,n_3}$ in \eqref{eq:dep2_phi} and \eqref{eq:dep3_phi} do not depend on choices of $\phi$ for depth $\ge2$.
We note that in the case when $n_1+n_2$ even, the coefficient of $f_{n_1+n_2}$ in \eqref{eq:dep2_phi} does depend on choices of $\phi$ for depth 2.

No attempt has been made here to generalize Proposition \ref{prop:formula_phi} for depth $r\ge4$.
The crucial difference is the fact that $\depth_{r-1} \mathcal{L}_m \neq \Q\xi_m$ holds in general.

\subsection{Choice of $\phi$ for depth 2}

One can make a choice of $\phi$ for depth 2 (note that the proof of our result on depth 2 does not need this choice).
From \eqref{eq:dep2_phi}, this is equivalent to determine rational numbers $\tau(n_1,n_2)\in\Q$ such that
\begin{equation}\label{eq:dep2_tau_def}
\phi(\zeta^\m(n_1,n_2)) = \sum_{\substack{m_1+m_2=n_1+n_2\\m_1\ge3;{\rm odd}\\m_2\ge2}} e\tbinom{m_1,m_2}{n_1,n_2} f_{m_1}f_{m_2}+\tau(n_1,n_2)f_{n_1+n_2}.
\end{equation}

We use the regularized double shuffle relations of multiple zeta values for depth 2 (see \cite{GKZ,IKZ}).
It is known that the motivic multiple zeta values satisfy the regularized double shuffle relations (see \cite[Theorem 7.4]{G2}).
Hence, the dimension formula \cite[Proposition 18]{IKZ} implies that all relations among motivic double zeta values and single zeta values are obtained from double shuffle relations
\begin{equation}\label{eq:ds}
\begin{aligned}
\zeta^\m(n_1)\zeta^\m(n_2) &= \zeta^\m(n_1,n_2)+\zeta^\m(n_2,n_1) + \zeta^\m(n_1+n_2) \\
&=\sum_{m_1+m_2=n_1+n_2} \left( \binom{m_2-1}{n_1-1}+\binom{m_2-1}{n_2-1}\right) \zeta^\m(m_1,m_2)
\end{aligned} \quad (n_1,n_2\ge1)
\end{equation}
and $\zeta^\m(n_1)\zeta^\m(n_2)=\frac{\beta_{n_1}\beta_{n_2}}{\beta_{n_1+n_2}}\zeta^m(n_1+n_2)$ for $n_1,n_2\ge2$ even, where for $k\in \Z_{\ge0}$, with the $k$th Bernoulli number $B_k$ we let 
\begin{equation}\label{beta} \beta_k=\begin{cases} -\frac{B_k}{2k!} &\mbox{if $k$ even} \\ 0 & \mbox{if $k$ odd} \end{cases}.
\end{equation} 
Since these relations are preserved by the map $\phi$, for $N=n_1+n_2$ even, applying $\phi$ to these relations and then using 
\[ \phi(\zeta^\m(n_1)\zeta^\m(n_2))=\begin{cases} f_{n_1}\shuffle f_{n_2} & \mbox{if }\ n_1,n_2\ge1 :\ {\rm odd} \\ 
\frac{\beta_{n_1}\beta_{n_2}}{\beta_{n_1+n_2}}f_{n_1+n_2} &   \mbox{if }\ n_1,n_2\ge2 :\ {\rm even}
\end{cases}\]
with $f_1=0$ and taking the coefficient of $f_N$, we get a system of linear equations satisfied by $\tau(n_1,n_2)$ as follows:
\begin{equation}\label{eq:ds_eq}
\begin{aligned}
\frac{\beta_{n_1}\beta_{n_2}}{\beta_N} &=\tau(n_1,n_2)+\tau(n_2,n_1) +1\\
&= \sum_{m_1+m_2=N} \left( \binom{m_2-1}{n_1-1}+\binom{m_2-1}{n_2-1}\right) \tau(m_1,m_2).
\end{aligned}
\end{equation}
A solution to the equations \eqref{eq:ds_eq} gives a choice of $\phi$ for depth 2.
A specific solution already appears in a work of Gangl, Kaneko and Zagier \cite{GKZ}.

\begin{proposition}\label{prop:tau_formula}
For integers $n_1,n_2\ge1$ with $N=n_1+n_2$ even, let
\begin{align*}
\tau(n_1,n_2)&= -\frac{1}{12} \left( 5+(-1)^{n_2} \binom{N-1}{n_2-1} - (-1)^{n_2} \binom{N-1}{n_2}\right)\\
& +\frac{\beta_{n_1}\beta_{n_2}}{3\beta_N} + \frac{(-1)^{n_2}}{3\beta_N} \sum_{j=2}^{N} \binom{j-1}{n_2-1}\beta_j \beta_{N-j}.
\end{align*}
Then, these are a solution to \eqref{eq:ds_eq}.
\end{proposition}

We will recall the proof of Proposition \ref{prop:tau_formula} in Appendix A.
Note that in \cite[\S6]{GKZ} the number $\tau(n_1,n_2)$ is called the Bernoulli realization of the formal double zeta space.
It is worth pointing out that there are other solutions to \eqref{eq:ds_eq} (see \cite[\S7.3]{B2} and also Appendix A). 

As is mentioned in \cite[\S7.4]{B2}, an explicit choice of $\phi$ can be applied to an expression of a chosen basis of motivic multiple zeta values, and also to linear relations of multiple zeta values.
For example, with Proposition \ref{prop:tau_formula}, one has
\[ \phi(\zeta^\m(2,4))=- \frac{4}{3}f_6+2 f_3f_3,\ \phi(\zeta^\m(4,2))=\frac{25}{12}f_6-2 f_3f_3.\]
Hence $\phi(\zeta^m(2,4)+\zeta^m(4,2))=\frac{3}{4}f_6=\phi(\frac{3}{4}\zeta^\m(6))$ holds, which by the injectivity of $\phi$ and the period map \eqref{eq:per} gives the relation $\zeta(2,4)+\zeta(4,2)=\frac{3}{4}\zeta(6)$.

\begin{remark}
We briefly mention a formula for $\tau(n_1,n_2)$ in the case when $n_1+n_2$ is odd.
Let $N$ be a positive odd integer.
Since $\phi(\zeta^\m(odd)\zeta^\m(even))=f_{odd}f_{even}$, the expression \eqref{eq:dep2_tau_def} for positive integers $n_1,n_2\ge1$ with $N=n_1+n_2$ gives the relation of the form
\begin{equation}\label{eq:dep2_parity}
 \zeta^\m(n_1,n_2) = \sum_{\substack{m_1+m_2=N\\ m_1\ge3:{\rm odd}\\m_2\ge2}} e\tbinom{m_1,m_2}{n_1,n_2} \zeta^\m(m_1)\zeta^\m(m_2)+\tau(n_1,n_2)\zeta^\m(N).
 \end{equation}
It can be shown that the set $\{ \zeta^\m(m_1)\zeta^\m(m_2)\mid m_1+m_2=N ,m_1\ge3:{\rm odd}, m_2\ge2\}$ forms a basis of the $\Q$-vector space spanned by motivic double zeta values of weight $N$.
Hence, the number $\tau(n_1,n_2)$ in \eqref{eq:dep2_parity} is uniquely determined (not depending on choices of $\phi$ for depth 2!).
Its explicit formula is obtained from the coefficient of $\zeta^\m(N)$ in the motivic version of Zagier's explicit formula for the parity theorem of depth 2 \cite[Proposition 7]{Z3}, so we have
\[ \tau(n_1,n_2) = \frac{(-1)^{n_1+1}}{2} \left( (-1)^{n_1}+\binom{N-1}{n_1-1} + \binom{N-1}{n_2-1}\right) \quad (N=n_1+n_2:{\rm odd}).\]
\end{remark}

\section{Double zeta values and period polynomials}

\subsection{Statement of results}

In this subsection, we state our results on depth 2.

Set
\[\zeta^{\frac12}(r,s) = \zeta(r,s)+\frac{1}{2} \zeta(r+s),\]
which is a special case of Yamamoto's $t$-interpolated multiple zeta values at $t=\frac{1}{2}$ (see \cite{Yamamoto}).
Let $S_k$ denote the $\C$-vector space of cusp forms of weight $k$ for ${\rm SL}_2(\Z)$. 
The even (resp. odd) period polynomial $P_f^+(x,y)$ (resp. $P_f^-(x,y)$) of a cusp form $f\in S_k$ is defined as a generating polynomial of critical values at odd integer points $s\in\{1,3,\ldots,k-1\}$ (resp. at even integer points $s\in \{2,4,\ldots,k-2\}$) of the completed $L$-function $L_f^\ast(s) = \int_0^{\infty} f(it)t^{s-1}dt$:
\begin{align*}
P_f^+(x,y) &= \sum_{\substack{r+s=k\\r,s\ge1:{\rm odd}}} (-1)^{\frac{s-1}{2}} \binom{k-2}{s-1} L_f^\ast(s) x^{r-1}y^{s-1}\\
(\mbox{resp.}\ P_f^-(x,y) &= \sum_{\substack{r+s=k\\r,s\ge1:{\rm even}}} (-1)^{\frac{s}{2}} \binom{k-2}{s-1} L_f^\ast(s) x^{r-1}y^{s-1}).
\end{align*}

\begin{theorem}\label{main1}
For $f\in S_k$, we define numbers $a_{i,j}(f)\in\C$ by
\[P_f^{+}(x+y,x)   = \sum_{i+j=k} \binom{k-2}{i-1} a_{i,j}(f) x^{i-1}y^{j-1}.\]
Then we have
\begin{equation}\label{eq:main1}
 \sum_{\substack{r+s=k\\r\ge1:{\rm odd}\\s\ge3:{\rm odd}}} a_{r,s}(f) \zeta^{\frac12} (r,s) =0 .
 \end{equation}
\end{theorem}

Since there is a generator $\{R_n \mid 1\le n\le k-1:{\rm odd}\}$ of the space $S_k$ such that $P_{R_n}^{+}\in \Q[x,y]$ (see \cite{KZ}), the relation in Theorem \ref{main1} can be over $\Q$.
The first example of even period polynomials is the one attached to the cusp form $\Delta=q\prod_{n=1}^{\infty} (1-q^n)^{24}$ of weight 12:
\begin{equation}\label{ex1}
c^{-1} P_\Delta^{+}(x,y)=\frac{36}{691} (x^{10} -y^{10}) - x^2y^2(x^2-y^2)^3, 
\end{equation}
where the constant $c$ is the coefficient of $x^2y^8$ in $P_\Delta^{+}(x,y)$.
For this, after multiplication by constant, Theorem \ref{main1} gives the relation of the form
\begin{equation*}\label{ex1_1}
\begin{aligned}
&22680\zeta^{\frac12}(1,11) +13006 \zeta^{\frac12} (3,9)-29145 \zeta^{\frac12}(5,7)\\
& -35364 \zeta^{\frac12}(7,5) +22680 \zeta^{\frac12}(9,3)=0.
\end{aligned}
\end{equation*}
We remark that the numbers $a_{i,j}(f)\in \C $ are written in the form
\begin{equation}\label{eq:formula_a} 
a_{i,j}(f) = \sum_{\substack{r+s=k\\r,s:{\rm odd}}} (-1)^{\frac{s-1}{2}} L_f^\ast(s) \binom{i-1}{s-1}, 
\end{equation}
and that $a_{k-1,1}(f)=P_f^+(1,1)=0$ holds for $f\in S_k$, because $P_f^+(x,y)+P_f^+(y,x)=0$ (which follows from the functional equation $L_f^\ast(s)=(-1)^{\frac{k}{2}} L_f^\ast (k-s)$).

\

We point out a difference from the result of Gangl, Kaneko and Zagier below.
Let $M_k$ denote the space of modular forms of weight $k$ for ${\rm SL}_2(\Z)$.
As a consequence of \cite[Theorem 3]{GKZ} (note that they use the opposite convention, i.e. their double zeta value $\zeta(r,s)$ equals our $\zeta(s,r)$), for $f\in M_k$ they obtained the relation of the form
\begin{equation}\label{eq:gkz} 
\sum_{\substack{r+s=k\\r,s:{\rm even}}} a_{r,s}(f) \zeta(r,s) = 3 \sum_{\substack{r+s=k\\r,s:{\rm odd}}} a_{r,s}(f) \zeta(r,s) + \sum_{r+s=k} (-1)^{r-1}a_{r,s}(f) \zeta(k),
\end{equation}
where the coefficient of $\zeta(k)$ in \eqref{eq:gkz} is described in the proof of \cite[Theorem 3]{GKZ}.
For the period polynomial of modular forms, we refer the reader to \cite{Z4}.
This shows that there is a linear relation corresponding to the Eisenstein series.
Our Theorem \ref{main1} does not apply to the Eisenstein series, but we emphasize that Theorem \ref{main1} gives a nontrivial simplification of the relation \eqref{eq:gkz} for the cusp form. 
Actually, since $a_{r,s}(f)=a_{s,r}(f)$ holds for $r,s$ even (see \cite[Theorem 3]{GKZ}), for a modular form $f\in M_k$ with rational periods we have the relation of the form
\begin{equation}\label{eq:gkz_ver}
\sum_{\substack{r+s=k\\r,s:{\rm odd}}} a_{r,s}(f) \zeta(r,s) \equiv 0 \pmod{\Q\zeta(k)},
\end{equation}
and the coefficient of $\zeta(k)$ in \eqref{eq:gkz_ver} is not as simple as all that. 
Theorem \ref{main1} provides a simpler formula for this, if $f$ is a cusp form.
We will see in the proof of Theorem \ref{main1} that the relation \eqref{eq:gkz} for $f\in S_k$ differs from Theorem \ref{main1} by the Kohnen-Zagier relation \cite[Theorem 9]{KZ}, an extra relation of $L_f^\ast(s) \ (1\le s\le k-1:{\rm odd})$.

\

On this occasion, let us recast an odd weight analogue of \eqref{eq:gkz} proved by the first author.
  
\begin{theorem}\label{thm:Ding-odd}
i) For $f\in S_k$, we define numbers $b_{i,j}(f)\in\C$ by
\[P_f^{-}(x+y,y) - \frac{x}{y} P_f^{-} (x+y,x)  = \sum_{i+j=k} \binom{k-1}{i-1} b_{i,j}(f) x^{i-1}y^{j-1}.\]
Then we have
\[ \sum_{\substack{r+s=k\\r,s\ge1:{\rm odd}}} b_{r,s}(f) \zeta^{\frac12} (r,s+1) =0 .\]
ii) For $f\in S_k$, we define numbers $c_{i,j}(f)\in\C$ by
\[\frac{d}{dx} P_f^{+}(x+y,y) - \frac{d}{dy} P_f^{+} (x+y,x)  = \sum_{i+j=k-1} \binom{k-3}{i-1} c_{i,j}(f) x^{i-1}y^{j-1}.\]
Then we have
\[ \sum_{\substack{r+s=k-1\\r\ge1:{\rm odd}\\s\ge2:{\rm even}}} c_{r,s}(f) \zeta^{\frac12} (r,s) =0 .\]
\end{theorem}

Theorem \ref{thm:Ding-odd} is a $\zeta^\frac12$-version of \cite[Theorems 1 and 2]{M1}, which was motivated by Zagier's discovery of a connection between cusp forms and double zeta values of odd weight \cite[\S6]{Z3}.
Theorem \ref{thm:Ding-odd} follows from a combination of \cite[Theorems 1 and 2]{M1} and \cite[Lemma 5.1]{M2}, so we omit the proof.

Let us give an example of Theorem \ref{thm:Ding-odd}.
The odd period polynomial of $\Delta$ is, up to constant, given by
\begin{equation}\label{eq:odd-period-poly-wt12}
4x^9y -25 x^7y^3 + 42x^5y^5 -25 x^3y^7+4xy^9.
\end{equation}
By Theorem \ref{thm:Ding-odd} i) this leads to the relation 
\[ -12 \zeta^{\frac12}(3,10) -14 \zeta^{\frac12}(5,8)+5 \zeta^{\frac12}(7,6) + 18 \zeta^{\frac12}(9,4)=0. \]
From Theorem \ref{thm:Ding-odd} ii) for the polynomial \eqref{ex1}, we have 
\[ 14 \zeta^{\frac12}(3,8) +10 \zeta^{\frac12}(5,6) -21 \zeta^{\frac12}(7,4)=0.\]

\subsection{Proof of Theorem \ref{main1}}

In this subsection, we prove a motivic version of Theorem \ref{main1}.
Theorem \ref{main1} is then obtained by the period map \eqref{eq:per}.

\begin{proof}[Proof of Theorem \ref{main1}]
Since the identity \eqref{eq:gkz} can be shown by using the double shuffle relation \eqref{eq:ds} (see \cite[Theorem 3]{GKZ}), it follows from a result of Goncharov \cite[Theorem 7.4]{G2} that the motivic version of \eqref{eq:gkz} holds. 
The motivic version of the relation \eqref{eq:gkz_ver} (which is a consequence of the motivic version of \eqref{eq:gkz}) says that for a cusp form $f\in S_k$ the image of $\sum_{\substack{r+s=k\\r,s:{\rm odd}}} a_{r,s}(f) \zeta^\m (r,s)$ under the map $\phi$ equals $\sum_{\substack{r+s=k\\r,s:{\rm odd}}} a_{r,s}(f) \tau (r,s)f_k\in \C f_k$, where the map $\phi$ is extended by $\C$-linearly.
With this, applying $\phi$ to the left-hand side of \eqref{eq:main1} (replace $\zeta$ with $\zeta^\m$), we get
\begin{equation}\label{eq:im_phi} 
\phi \left( \sum_{\substack{r+s=k\\r,s:{\rm odd}}} a_{r,s}(f) \left(\zeta^\m (r,s)+\frac{1}{2} \zeta^\m(k)\right) \right) = \sum_{\substack{r+s=k\\r,s:{\rm odd}}} a_{r,s}(f) \left(\tau (r,s)+\frac{1}{2} \right)  f_k.
\end{equation}
Therefore we only need to show that the coefficient of $f_k$ in the right hand-side of \eqref{eq:im_phi} is zero.
This is done as follows.

Since $\zeta^\m(r,s)+\zeta^\m(s,r)=\left(\frac{\beta_r\beta_s}{\beta_{r+s}} -1\right)\zeta^\m(r+s)$ and the symmetry $a_{r,s}(f)=a_{s,r}(f)$ hold for $r,s\ge2$ even, the relation \eqref{eq:gkz} can be reduced to
\[ \sum_{\substack{r+s=k\\r,s:{\rm even}}} a_{r,s}(f) \left( \frac{\beta_r \beta_s}{\beta_k} + 1 \right) \zeta^\m(k) = 6\sum_{\substack{r+s=k\\r,s:{\rm odd}}} a_{r,s}(f) \left( \zeta^\m(r,s)+ \frac13 \zeta^\m(k)\right).\]
Applying $\phi$ and then comparing the coefficient of $f_k$, one gets the identity of the form
\begin{equation}\label{eq:dep2_pr1} \sum_{\substack{r+s=k\\r,s:{\rm even}}} a_{r,s}(f) \left( \frac{\beta_r \beta_s}{\beta_k} + 1 \right) = 6\sum_{\substack{r+s=k\\r,s:{\rm odd}}} a_{r,s}(f) \left(\tau(r,s)+ \frac13\right).
\end{equation}
We now reduce \eqref{eq:dep2_pr1} into the identity \eqref{eq:goal} below.
For this, we only need a result of Kohnen and Zagier \cite{KZ}.
For $r,s\ge1$ with $k=r+s$ even, let
\begin{align*}
\lambda(r,s)&=-\frac{1}{12} \left( 1 -(-1)^{s} \binom{k-1}{s-1}+(-1)^{s} \binom{k-1}{s} \right)-\frac{(-1)^{s}}{3\beta_k} \sum_{j=2}^k \binom{j-1}{s-1} \beta_j \beta_{k-j}.
\end{align*}
In particular, for $r,s\ge1$ odd, we have 
\begin{equation}\label{eq:lambda_odd} 
\lambda(r,s)=-\frac{1}{12} \left( 1+\binom{k-1}{s-1}- \binom{k-1}{s} \right)+\frac{1}{3\beta_k} \sum_{j=2}^k \binom{j-1}{s-1} \beta_j \beta_{k-j},
\end{equation}
which coincides with $-\frac{\lambda_{k/2,s-1}}{12 B_k}$, where $\lambda_{k,n}$ is defined in \cite[Theorem 9]{KZ}.
With this, as a consequence of Harberland's formula for the Petersson inner product between the Eisenstein series and cusp forms, Kohnen and Zagier \cite[Theorem 9 (ii)]{KZ} show that the relation
\begin{equation}\label{eq:KZ_relation}
 \sum_{\substack{r+s=k\\r,s:{\rm odd}}}(-1)^{\frac{s-1}{2}} \lambda(r,s) L_f^\ast(s) =0 
 \end{equation}
holds for any $f\in S_k$.
Using the functional equation $L_f^\ast(s)=(-1)^{\frac{k}{2}}L_f^\ast(k-s)$, one can reduce the Kohnen-Zagier relation to the form
\begin{equation}\label{eq:dep2_pr2} 
\sum_{\substack{r+s=k\\r,s:{\rm even}}} a_{r,s}(f) \left( \frac{\beta_r\beta_s}{\beta_k}+1\right) +  \sum_{\substack{r+s=k\\r,s:{\rm odd}}} a_{r,s}(f) =0 \quad (\forall f\in S_k).
\end{equation}
The details are left to Appendix B.
Combining \eqref{eq:dep2_pr1} and \eqref{eq:dep2_pr2} gives the relation
\begin{equation}\label{eq:goal} 
\sum_{\substack{r+s=k\\r,s:{\rm odd}}} a_{r,s}(f) \left(\tau (r,s)+\frac{1}{2} \right)=0,
\end{equation}
which completes the proof.
\end{proof}

Notice that the above proof works for any choice of the map $\phi$ for depth 2.
We emphasize that $\tau(r,s)$ in Proposition \ref{prop:tau_formula} has a close connection to $\lambda(r,s)$ in the Kohnen-Zagier relation.
In fact, we easily see that
\[ \lambda(r,s )=- \tau(r,s)-\frac{1}{2} + \frac{\beta_{r}\beta_{s}}{3\beta_{r+s}}.\]
This suggests a further connection between solutions to the double shuffle equation \eqref{eq:ds_eq} and extra relations of $L_f^\ast(odd)$'s which comes from the orthogonality between the Eisenstein series and $f$.
Another spin-off is that we obtain the anti-symmetry $\lambda(r,s)=-\lambda(s,r)$ for $r,s$ odd, first proved in \cite[Theorem 9 (i)]{KZ}, as a consequence of the first equality of \eqref{eq:ds_eq}.

\section{Triple zeta values and period polynomials}
\subsection{Statements of results}
In this subsection, we state our results on depth 3 and describe their consequences. 

For $j\in \{1,2,3\}$ we denote by $\mathbb{I}_k^{(j)}$ the set of $j$-th almost totally odd indices of weight $k$ and depth 3:
\[ \mathbb{I}_k^{(j)}=\{ (k_1,k_2,k_3)\in\Z_{\ge2}^3\mid k_1+k_2+k_3=k, k_j:{\rm even}, k_i:{\rm odd} \ i\neq j\}. \]
For example, $\mathbb{I}_{10}^{(3)}=\{(3,5,2),(3,3,4),(5,3,2)\}$.
Let $d_k$ be the number of elements in $\mathbb{I}_k^{(j)}$ (which does not depend on $j$).
For each $k$ even and $j\in\{1,2,3\}$, we define the $d_k\times d_k$ matrix $C_k^{(j)}$ by
\[ C_k^{(j)} = \left( c\tbinom{\bf m}{\bf n} \right)_{\begin{subarray}{c} {\bf m} \in \mathbb{I}_k^{(3)} \\ {\bf n} \in \mathbb{I}_k^{(j)} \end{subarray}}, \]
whose rows and columns are indexed by ${\bf m}$ and ${\bf n}$ in the sets $\mathbb{I}_k^{(3)}$ and $\mathbb{I}_k^{(j)}$, respectively.

The first result is about right annihilators of the square matrix $C_k^{(j)} $.

\begin{theorem}\label{thm:rankC}
Let $k$ be a positive even integer.
For $j\in\{1,2,3\}$, the relation
\[ \sum_{(k_1,k_2,k_3)\in \mathbb{I}_k^{(j)}} a_{k_1,k_2,k_3} \zeta^\m (k_1,k_2,k_3)\in \mathfrak{D}_2\mathcal{H}_k \]
holds if and only if the column vector $(a_{n_1,n_2,n_3})_{(n_1,n_2,n_3)\in \mathbb{I}_k^{(j)}}$ is a right annihilator of the matrix $C_k^{(j)}$.
Furthermore, we have
\begin{equation}\label{eq:rankC} 
{\rm rank}\, C_k^{(j)} = \dim_\Q  \langle \zeta^\m(k_1,k_2,k_3) \mod \mathfrak{D}_2\mathcal{H}_k\mid (k_1,k_2,k_3)\in\mathbb{I}_k^{(j)} \rangle_\Q.
\end{equation}
\end{theorem}

Hereafter, we call each elements $\zeta^\m(odd_{\ge3},odd_{\ge3},even_{\ge2})$, $\zeta^\m(odd_{\ge3},even_{\ge2},odd_{\ge3})$ and $\zeta^\m(even_{\ge2},odd_{\ge3},odd_{\ge3})$ the $j$-th almost totally odd motivic triple zeta values, where $j$ indicates the position of $even_{\ge2}$ in the indices.
Theorem \ref{thm:rankC} says that right annihilators of $C_k^{(j)}$ give all linear relations of the $j$-th almost totally odd motivic triple zeta values of weight $k$ modulo lower depths.

\

The second results are about various connections between left (not right!) annihilators of the square matrix $C_k^{(j)}$ and period polynomials.
These results have applications to giving upper bounds of the dimension of the $\Q$-vector space spanned by $j$-th totally odd (motivic) triple zeta values.
A similar study for totally odd multiple zeta values is examined by the second author in \cite{T}. 

Before stating the second results, we begin with conjectural formulas for the generating series of $\dim_\Q\ker C_k^{(j)}$, where the space $\ker C_k^{(j)}$ denotes the $\Q$-vector space of left annihilators of the matrix $C_k^{(j)}$.
We let 
\[ \odd(x) = \frac{x^3}{1-x^2}=x^3+x^5+\cdots,\ \even(x)=\frac{x^2}{1-x^2}=x^2+x^4+\cdots\]
and
\[ \cusp(x)=\frac{x^{12}}{(1-x^4)(1-x^6)}=\sum_{k>0} \dim_\C S_k x^k.\]
Note that $\sum_{k>0:{\rm even}}{\rm rank}\, C_k^{(j)}  x^k =  \odd(x)^2\even(x) - \sum_{k>0:{\rm even}}  \dim_\Q \ker C_k^{(j)} x^k $.

\begin{conjecture}\label{conj:dim_C^j}
We have
\begin{align*}
 \sum_{k>0:{\rm even}} \dim_\Q \ker C_k^{(1)}x^k &\stackrel{?}{=} \frac{1}{x^2} \cusp (x)\even(x) ,\\
 \sum_{k>0:{\rm even}} \dim_\Q \ker C_k^{(2)}x^k &\stackrel{?}{=} \cusp (x)\even(x) ,\\
 \sum_{k>0:{\rm even}} \dim_\Q \ker  C_k^{(3)} x^k& \stackrel{?}{=} \frac{1}{x^2}  \cusp (x)\even(x) + (x+\frac{1}{x} )  \cusp (x)\odd(x).
\end{align*}
\end{conjecture}

We have checked the above equalities by Mathematica up to $k=40$.

By the parity theorem, for $k$ even we see that the space $\mathfrak{D}_3\mathcal{H}_k$ is generated by elements $\zeta^\m (n_1,n_2)\zeta^\m(2n)   \ (n_1+n_2+2n=k, n_1,n_2\ge1,n\ge0)$.
It is known that $\sum_{k>0:{\rm even}} \dim\big(\mathfrak{D}_2\mathcal{H}_k\big/\mathfrak{D}_1\mathcal{H}_k \big)x^k = \odd(x)^2-\cusp(x)$ (see \cite[Theorem 2.4]{G1} and \cite[Proposition 18]{IKZ}).
Therefore we have
\[ \sum_{k>0:{\rm even}} \dim_\Q \big( \mathfrak{D}_3\mathcal{H}_k / \mathfrak{D}_2\mathcal{H}_k \big)x^k = \odd(x)^2\even(x)- \cusp (x)\even(x).\]
This implies that  
\begin{equation}\label{eq:triv_ineq} 
\sum_{k>0:{\rm even}} \dim_\Q \ker C_k^{(j)}x^k\ge \cusp (x)\even(x)=x^{14}+x^{16}+\cdots \quad (j\in \{1,2,3\}),
\end{equation}
where $\sum a_k x^k \ge \sum b_k x^k$ means $a_k\ge b_k$ for all $k$.
Conjecture \ref{conj:dim_C^j} implies that the spaces $\ker C_k^{(j)}$ for $j\in\{1,3\}$ will have more elements which are not obtained from the dimension formula for $\mathfrak{D}_2\mathcal{H}_n$.

\

Let us state the second results.
We begin with an alternative result to \eqref{eq:triv_ineq}.
Denote by $W_k^{+,0}$ the space of restricted even period polynomials defined as a subspace of $\Q[x_1,x_2]$ of homogeneous degree $k-2$ such that $p\in W_k^{+,0}$ satisfies $p(x_1,x_2)=p(-x_1,x_2)$ (even polynomial), $p(x_1,0)=0$ and
\[ p(x_1,x_2)-p(x_1+x_2,x_2)+p(x_1+x_2,x_1)=0.\]
Let $P_k^{+}$ be the $\Q$-vector space spanned by the polynomials
\[ p(x_1,x_2)x_3^{k-n-1} \quad (p(x_1,x_2)\in W^{+,0}_n ,\ 0< n < k  ). \]
It follows by definition that for $k$ even one has
\[P_k^{+}\cong \bigoplus_{\substack{1 < n< k\\ n:{\rm even}}} \big( W^{+,0}_n \otimes_\Q  \Q x^{k-n-1} \big). \]

\begin{theorem}\label{thm:rest_even_C^j}
For $j\in\{1,2,3\}$ and $k$ even, the map
\begin{align*}
P_k^{+} &\longrightarrow \ker C_k^{(j)}\\
\sum_{{\bf n}\in \mathbb{I}_k^{(3)}} a_{\bf n}{\bf x}^{\bf n}&\longmapsto  (a_{\bf n})_{{\bf n}\in \mathbb{I}_k^{(3)}}
\end{align*}
is well-defined and injective, where we write ${\bf x}^{\bf n}=x_1^{n_1-1}x_2^{n_2-1}x_3^{n_3-1}$ for ${\bf n} =(n_1,n_2,n_3)$.
Namely, the row vector obtained from coefficients of $p\in P_k^+$ gives a left annihilator of the matrix $C_k^{(j)}$.
\end{theorem}

Since $W_k^{+,0}\otimes \C \cong S_k$ (see e.g.~\cite[\S1.1]{BS}), we have $\sum \dim_\Q P_k^{+} x^k = \cusp(x)\even(x)$.
Thus, the inequality \eqref{eq:triv_ineq} is also obtained from Theorem \ref{thm:rest_even_C^j}.
It is remarkable that our proof of Theorem \ref{thm:rest_even_C^j} has an application to an explicit formula for the parity theorem for $j$-th almost totally odd motivic triple zeta values (see Theorem \ref{9_1}).


\

We now turn to a result for the case $j=3$.
By definition, we easily see that $C_k^{(3)} = B_k^{(3)} E_k^{(3)}$, where
\begin{align*}
B_k^{(3)} &= \left( \delta\tbinom{m_1}{n_1}e\tbinom{m_2,m_3}{n_2,n_3} \right)_{\begin{subarray}{c} (m_1,m_2,m_3)\in \mathbb{I}_k^{(3)}\\ (n_1,n_2,n_3)\in \mathbb{I}_k^{(3)}\end{subarray}},\\
E_k^{(3)} &= \left( e\tbinom{m_1,m_2,m_3}{n_1,n_2,n_3} \right)_{\begin{subarray}{c} (m_1,m_2,m_3)\in \mathbb{I}_k^{(3)}\\ (n_1,n_2,n_3)\in \mathbb{I}_k^{(3)} \end{subarray}}
\end{align*}
are the square matrices whose rows and columns are indexed by $(m_1,m_2,m_3)$ and $(n_1,n_2,n_3)$ in the set $ \mathbb{I}_k^{(3)} $, respectively.
Note that the square matrix whose rows and columns are indexed by ${\bf m}$ and ${\bf n}$ in the set $ \mathbb{I}_k^{(3)} $ is naturally viewed as a linear map on the $\Q$-vector space spanned by row vectors $(a_{\bf n})_{{\bf n}\in \mathbb{I}_k^{(3)}}$ indexed by $\mathbb{I}_k^{(3)}$ with $a_{\bf n}\in \Q$.
Then, by linear algebra, we have
\begin{equation}\label{eq:decomp_C^3}
\dim \ker C_k^{(3)} =\dim \ker B_k^{(3)} +\dim ({\rm Im}\, B_k^{(3)} \cap \ker E_k^{(3)} ).
\end{equation}
We now describe connections between period polynomials and row vectors in $\ker B_k^{(3)}$ and $ {\rm Im}\, B_k^{(3)} \cap \ker E_k^{(3)}$, separately.

The space $\ker B_k^{(3)}$ involves both odd and restricted even period polynomials.
Let $k$ be a positive even integer.
Define the space $W_k^-$ of the odd period polynomials as a subspace of $\Q[x_1,x_2]$ of homogeneous degree $k-2$ such that $p\in W_k^-$ satisfies $p(x_1,x_2)=-p(-x_1,x_2)$ (odd polynomial) and
\[ p(x_1,x_2)-p(x_1+x_2,x_2)-p(x_1+x_2,x_1)=0.\]
We define a subspace $Q_k^+ \subset \Q[x_1,x_2,x_3]$ (resp. $Q_k^- \subset \Q[x_1,x_2,x_3]$) as the $\Q$-vector space spanned by $x_1^{n-1}p(x_2,x_3)$ for $p(x_1,x_2)\in W^{+,0}_{k-n+1}$ and $1<n<k$ odd (resp. $x_1^{n-1}p(x_2,x_3)$ for $p(x_1,x_2)\in W^{-}_{k-n-1}$ and $1<n<k$ odd).
It follows for $k$ even that
\begin{align*}
&Q_k^{+} \cong\bigoplus_{\substack{1 < n< k\\ n:{\rm odd}}} \big( \Q x^{n-1} \otimes_\Q  W^{+,0}_{k-n+1}   \big),\\
&Q_k^- \cong\bigoplus_{\substack{1 < n< k\\ n:{\rm odd}}} \big(   \Q x^{n-1} \otimes_\Q W^{-}_{k-n-1}  \big)
\end{align*}
and, by $W_k^{-}\otimes \C \cong S_k$ (see e.g.~\cite[\S1.1]{KZ}), we have $\sum \dim_\Q Q_k^{\pm} x^k = \odd(x)\cusp(x)x^{\pm1}$.
For simplicity of notation, for any subsets $S_1,\ldots,S_r$ of $\Z$, we let
\begin{equation}\label{eq2_5}
\mathbb{I}_k(S_1\cdots S_r) = \{ (n_1,\ldots,n_r)\in S_1\times \cdots \times S_r\mid k=n_1+\cdots+n_r\}.
\end{equation}
We denote by $\mathbf{o}$ (resp. $\mathbf{e}$) the set of all odd integers $>1$ (resp. all even integers $>1$).
For example, it follows that $\mathbb{I}_k^{(3)}=\mathbb{I}_k(\mathbf{ooe}) $.

\begin{theorem}\label{thm:B^3}
For $k$ even the maps
\begin{equation*}
\begin{minipage}{.45\textwidth}
\begin{align*}
Q_k^{+} &\longrightarrow \ker B_k^{(3)}\\
\sum_{{\bf n}\in \mathbb{I}_{k+1}(\mathbf{ooo})} a_{\bf n}{\bf x}^{\bf n}&\longmapsto  (a_{\bf n}^+)_{{\bf n}\in \mathbb{I}_k^{(3)}}
\end{align*}
\end{minipage}
and
\begin{minipage}{.45\textwidth}
\begin{align*}
Q_k^{-} &\longrightarrow \ker B_k^{(3)}\\
\sum_{{\bf n}\in \mathbb{I}_{k-1}(\mathbf{oee})} a_{\bf n}{\bf x}^{\bf n}&\longmapsto  (a_{\bf n}^-)_{{\bf n}\in \mathbb{I}_k^{(3)}}
\end{align*}
\end{minipage}
\end{equation*}
are well-defined and injective, where $a_{\bf n}^+=n_3 a_{n_1,n_3+1,n_2}$ and $a_{\bf n}^-=a_{n_1,n_2-1,n_3}$ for ${\bf n}=(n_1,n_2,n_3)$.
Furthermore, their combined map $Q_k^+\oplus Q_k^- \rightarrow \ker B_k^{(3)}$ is an injection.
\end{theorem}

For elements of $ {\rm Im}\, B_k^{(3)} \cap \ker E_k^{(3)}$, we define the subspace $\widehat{P}_k^+\subset x_3^{-1}\Q[x_1,x_2,x_3]$ as the $\Q$-vector space spanned by $p(x_1,x_2)x_3^{k-n-1}$ for $p(x_1,x_2)\in  W_n^{+,0}$ and $0<n\le k$ even.
For $k$ even we have
\[ \widehat{P}_k^{+} \cong\bigoplus_{\substack{0 < n\le k\\ n:{\rm even}}} \big( W^{+,0}_n \otimes_\Q  \Q x^{k-n-1} \big)\]
and hence, $\sum \dim_\Q \widehat{ P}_k^{+} x^k = \cusp(x)\even(x)x^{-2}$.

\begin{theorem}\label{thm:B^3+E^3}
There is a well-defined linear map from $\widehat{P}_k^+$ to $ {\rm Im}\, B_k^{(3)} \cap \ker E_k^{(3)}$.
\end{theorem}

Unlike the cases of Theorems \ref{thm:rest_even_C^j} and \ref{thm:B^3}, we were not able to prove that the well-defined map in Theorem \ref{thm:B^3+E^3} is injective.
Assuming the injectivity, from \eqref{eq:decomp_C^3} and Theorems \ref{thm:B^3} and \ref{thm:B^3+E^3}, we obtain the following inequality: 
\[ \sum_{k>0:{\rm even}} \dim_\Q \ker  C_k^{(3)} x^k \ge \frac{1}{x^2}  \cusp (x)\even(x) + (x+\frac{1}{x} )  \cusp (x)\odd(x)=x^{12}+2x^{14}+\cdots.\]
The result suggests that linear relations among $\zeta(odd_{\ge3},odd_{\ge3},even_{\ge2})$'s modulo lower depths may be related to cusp forms in three ways.

According to Conjecture \ref{conj:dim_C^j}, we should have further elements in $\ker C_k^{(1)}$ which does not come from Theorem \ref{thm:rest_even_C^j}.
For this, we observe a conjectural relation with a derivative of an odd period polynomial (see Section 4.6 below).
This could be an interesting phenomena since this is the first appearance of the derivative of odd period polynomials in this study.

\subsection{Proof of Theorem \ref{thm:rankC}}

In this subsection, we prove Theorem \ref{thm:rankC}.

\begin{proof}[Proof of Theorem \ref{thm:rankC}]
Let $(a_{\bf n})_{{\bf n}\in \mathbb{I}_k^{(j)}}$ be a right annihilator of the matrix $ C_k^{(j)}$ (a column vector).
Set
\[\xi=\sum_{{\bf n}\in \mathbb{I}_k^{(j)}} a_{\bf n} \zeta^\m ({\bf n}).\]
Since the element $\xi$ is of weight $k$ even and depth 3, by the parity theorem there are $b_{k_1,k_2,k_3}\in\Q$ and $\xi' \in \depth_2 \Zsp_k$ such that
\begin{equation*}\label{eq6_2}
\xi=\sum_{\substack{k_1+k_2+k_3=k\\k_3\ge2:{\rm even}}}b_{k_1,k_2,k_3} \zeta^\m (k_1,k_2) \zeta^\m (k_3) + \xi'.
\end{equation*}
From \eqref{eq:dep2_phi} we see that for $k_3$ even the element 
\begin{equation}\label{eq:image_phi}
\phi\bigg( \sum_{k_1+k_2=k-k_3} b_{k_1,k_2,k_3} \zeta^\m (k_1,k_2) \zeta^\m (k_3)\bigg)  \quad \big(= \phi(\xi-\xi')\big)
\end{equation}
lies in the $\Q$-vector space spanned by $f_{odd}f_{odd}f_{k_3}$'s and $f_k$.
On the other hand, since the column vector $(a_{\bf n})_{{\bf n}\in \mathbb{I}_k^{(j)}}$ is a right annihilator of $ C_k^{(j)}$, it follows from \eqref{eq:dep3_phi} that $\phi (\xi)\in \mathcal{U}_{k,2}$.
We also have $\phi (\xi')\in \mathcal{U}_{k,2}$ from \eqref{eq:dep2_phi}, and hence, $\phi (\xi-\xi')\in \mathcal{U}_{k,2}$.
Thus, the element \eqref{eq:image_phi} lies in the intersection of $\mathcal{U}_{k,2}$ and the space spanned by $f_{odd}f_{odd}f_{k_3}$'s and $f_k$, which is $\Q f_k$.
Thus, $\xi-\xi'\in \Q\zeta^\m(k)$, which implies $\xi\in \depth_2 \Zsp_k$.

Conversely, suppose that for $a_{\bf n}\in\Q $ we have
\[\xi=\sum_{{\bf n}\in \mathbb{I}_k^{(j)}} a_{\bf n} \zeta^\m ({\bf n})  \in \depth_2 \Zsp_k.\]
We see that $\big({\rm id}\otimes D_{m_2}\big) \circ D_{m_1}  (\zeta^\m(n,k-n))=0$ holds for all $(m_1,m_2,m_3)\in \mathbb{I}_k^{(3)}$ and $1\le n\le k-2$, which shows $\big({\rm id}\otimes D_{m_2}\big) \circ D_{m_1}(\xi)=0$.
On the other hand, using Proposition \ref{prop:D_m} and Definition \ref{def:c}, we get 
\[ \big({\rm id}\otimes D_{m_2}\big) \circ D_{m_1}  (\zeta^\m(n_1,n_2,n_3)) = c\tbinom{m_1,m_2,m_3}{n_1,n_2,n_3}\xi_{m_1}\otimes \xi_{m_2}\otimes \zeta^\m(m_3) .\]
Since the set $\{\xi_{m_1}\otimes \xi_{m_2}\otimes \zeta^\m(m_3) \mid (m_1,m_2,m_3)\in \mathbb{I}_k^{(3)}\}$ is linearly independent over $\Q$, the identity $\big({\rm id}\otimes D_{m_2}\big) \circ D_{m_1}(\xi)=0$ implies that the column vector $(a_{\bf n})_{{\bf n}\in \mathbb{I}_k^{(3)}}$ is a right annihilator of  the matrix $C_k^{(j)}$. 
We complete the proof.
\end{proof}

\subsection{Proof of Theorem \ref{thm:rest_even_C^j}}

In this subsection, we first give another expression of the integer $c\tbinom{m_1,m_2,m_3}{n_1,n_2,n_3}$, and then, prove Theorem \ref{thm:rest_even_C^j}.
We also show an explicit formula for the parity theorem of depth 3 modulo lower depths.

For a rational function $ f(x_1,x_2,x_3)\in\Q(x_1,x_2,x_3)$, define the change of variables $\sigma_{ij}$ for $1\le i\le 5,1\le j \le2$ by
\begin{equation*}
\begin{aligned}
& f(x_1,x_2,x_3) \big| \sigma_{11}=f(x_2-x_1,x_1,x_3),\, f(x_1,x_2,x_3) \big| \sigma_{12}=  f(x_2-x_1,x_2,x_3),\\
& f(x_1,x_2,x_3) \big| \sigma_{21}=f(x_3-x_2,x_1,x_2),\, f(x_1,x_2,x_3) \big| \sigma_{22} =f(x_3-x_2,x_1,x_3),\\
 & f(x_1,x_2,x_3) \big| \sigma_{31}=f(x_1,x_3-x_2,x_2),\, f(x_1,x_2,x_3) \big| \sigma_{32}= f(x_1,x_3-x_2,x_3),\\
& f(x_1,x_2,x_3) \big| \sigma_{41}=f(x_2-x_1,x_3-x_1,x_1),\, f(x_1,x_2,x_3) \big| \sigma_{42} = f(x_1-x_2,x_3-x_2,x_2),\\
&  f(x_1,x_2,x_3) \big| \sigma_{51}=f(x_3-x_2,x_3-x_1,x_3),\, f(x_1,x_2,x_3) \big| \sigma_{52}= f(x_2-x_3,x_2-x_1,x_2).
 \end{aligned}
 \end{equation*}
For $a,b\in \Z$ we write $f\big| (a\sigma_{ij}+b\sigma_{kl})=af\big| \sigma_{ij}+bf\big|\sigma_{kl}$ and set
\begin{align}\label{eq7_1}
f\big|\sigma_i:=f\big|(\sigma_{i1}-\sigma_{i2}) \quad (1\le i\le 5).
\end{align}

\begin{lemma}\label{7_1}
For any rational function $f(x_1,x_2,x_3)\in\Q(x_1,x_2,x_3)$ satisfying $f(\pm x_1,\pm x_2,x_3)=f(x_1,x_2,x_3)$, we have
\begin{equation}\label{eq7_2}
(f\big|(1+\sigma_3))\big| (1+\sigma_1+\sigma_2) = (f\big| (1+\sigma_1)) \big| (1+\sigma_2+\sigma_3+\sigma_4+\sigma_5),
\end{equation}
where $f\big|(1+\sigma_i)$ means $f+f\big| \sigma_i$.
\end{lemma}

\begin{proof}
One computes
\begin{align*}
f\big|\sigma_1\big|\sigma_2&= \big( f(x_2-x_1,x_1,x_3)-  f(x_2-x_1,x_2,x_3)\big)\big| \sigma_2\\
&=  \big( f(x_1-x_2,x_1,x_3)-  f(x_1-x_2,x_2,x_3)\big)\big| \sigma_2\\
&= \big( f((x_3-x_2)-x_1,x_3-x_2,x_2)-  f((x_3-x_2)-x_1,x_1,x_2)\big) \\
&- \big( f((x_3-x_2)-x_1,x_3-x_2,x_3)-  f((x_3-x_2)-x_1,x_1,x_3)\big) \\
&=-\big( f((x_3-x_2)-x_1,x_1,x_2)-  f((x_3-x_2)-x_1,x_3-x_2,x_2)\big)\\
&+\big( f((x_3-x_2)-x_1,x_1,x_3)-  f((x_3-x_2)-x_1,x_3-x_2,x_3)\big)\\
&=-\big( f(x_2-x_1,x_1,x_3)-  f(x_2-x_1,x_2,x_3)\big)\big| \sigma_3\\
&=-f\big|\sigma_1\big|\sigma_3,
\end{align*}
where for the second equality we have used $f(-x_1,x_2,x_3)=f(x_1,x_2,x_3)$.
With this, the equation \eqref{eq7_2} is reduced to
\begin{align*}
 f\big|\sigma_3\big|\sigma_1 +f\big|\sigma_3\big|\sigma_2  &= f\big|\sigma_4+f\big|\sigma_5+f\big|\sigma_1\big|\sigma_4+f\big|\sigma_1\big|\sigma_5.
 \end{align*}
This equality follows from the easily checked identities
\begin{align*}
&f\big|\sigma_{31}\big|\sigma_{11} = f\big| \sigma_{41} ,\ f\big|\sigma_{31}\big|\sigma_{12} = f\big| \sigma_{42} , f\big|\sigma_{32}\big|\sigma_{11} = f\big| \sigma_{12}\big|\sigma_{51},\ f\big|\sigma_{32}\big|\sigma_{12} = f\big| \sigma_{11}\big|\sigma_{51},\\
& f\big|\sigma_{31}\big|\sigma_{21} = f\big| \sigma_{11}\big|\sigma_{41},\  f\big|\sigma_{31}\big|\sigma_{22} = f\big| \sigma_{12}\big|\sigma_{41},\ f\big|\sigma_{32}\big|\sigma_{21} = f\big| \sigma_{52} ,\  f\big|\sigma_{32}\big|\sigma_{22} = f\big| \sigma_{51} 
\end{align*}
and 
\[f\big|\sigma_{11}\big|\sigma_{42}=f\big|\sigma_{12}\big|\sigma_{52},\ f\big| \sigma_{12}\big|\sigma_{42}=f\big| \sigma_{11}\big|\sigma_{52},\]
which are valid for $f\in\Q(x_1,x_2,x_3)$ such that $f(\pm x_1,\pm x_2,x_3)=f(x_1,x_2,x_3)$.
\end{proof}

Comparing the coefficients of both sides of the identity \eqref{eq7_2}, one gets a new expression of the integer $c\tbinom{m_1,m_2,m_3}{n_1,n_2,n_3}$.
For integers $m_1,m_2,m_3,n_1,n_2,n_3\ge1$, let us define an integer $h\tbinom{m_1,m_2,m_3}{n_1,n_2,n_3}$ by the formula
\begin{equation}\label{eq7_3}
\begin{aligned}
&h\tbinom{m_1,m_2,m_3}{n_1,n_2,n_3} = \delta\tbinom{m_1,m_2,m_3}{n_1,n_2,n_3} + \delta\tbinom{m_2}{n_1} b_{n_2,n_3}^{m_1}+\delta\tbinom{m_1}{n_1} b_{n_2,n_3}^{m_2}\\
&+(-1)^{m_1+m_2+n_3} \binom{m_2-1}{n_3-1}\left((-1)^{n_2} \binom{m_1-1}{n_2-1} - (-1)^{n_1} \binom{m_1-1}{n_1-1}\right)\\
&+(-1)^{n_1} \binom{m_2-1}{n_1-1} \left((-1)^{n_2}  \binom{m_1-1}{n_2-1}- (-1)^{n_3} \binom{m_1-1}{n_3-1}\right).
\end{aligned}
\end{equation}
Write $\mathbf{a}=\mathbf{o}\cup \mathbf{e}$, meaning the set of all integers $>1$.

\begin{corollary}\label{7_2}
Let $k$ be a positive even integer and $j\in\{1,2,3\}$.
For all pairs $(m_1,m_2,m_3)\in \mathbb{I}_k^{(3)}$ and $(n_1,n_2,n_3)\in \mathbb{I}_k^{(j)}$, we have
\begin{equation}\label{eq:c_deform} 
c\tbinom{m_1,m_2,m_3}{n_1,n_2,n_3} = \sum_{(k_1,k_2,k_3)\in \mathbb{I}_k (\mathbf{aae}) } e\tbinom{m_1,m_2}{k_1,k_2}\delta\tbinom{m_3}{k_3} h\tbinom{k_1,k_2,k_3}{n_1,n_2,n_3}.
\end{equation}
\end{corollary}

\begin{proof}
One can easily verify that the coefficient of $x_1^{n_1-1}x_2^{n_2-1}x_3^{n_3-1}$ in
\[ x_1^{m_1-1}x_2^{m_2-1}x_3^{m_3-1}\big| (1+\sigma_1+\sigma_2)\quad \mbox{ and} \quad x_1^{m_1-1}x_2^{m_2-1}x_3^{m_3-1}\big| (1+\sigma_3)\]
equals $e\tbinom{m_1,m_2,m_3}{n_1,n_2,n_3}$ (see \eqref{eq:def_e3}) and $\delta\tbinom{m_1}{n_1}e\tbinom{m_2,m_3}{n_2,n_3}$, respectively.
Thus, by Definition \ref{def:c}, the coefficient of $x_1^{n_1-1}x_2^{n_2-1}x_3^{n_3-1}$ in $\big(x_1^{m_1-1}x_2^{m_2-1}x_3^{m_3-1}\big| (1+\sigma_3) \big) \big| (1+\sigma_1+\sigma_2)$ equals $c\tbinom{m_1,m_2,m_3}{n_1,n_2,n_3}$.
Likewise, one can check that the coefficient of $x_1^{n_1-1}x_2^{n_2-1}x_3^{n_3-1}$ in 
\[ x_1^{m_1-1}x_2^{m_2-1}x_3^{m_3-1}\big| (1+\sigma_1)\]
and 
\begin{align*}
&x_1^{m_1-1}x_2^{m_2-1}x_3^{m_3-1}\big|\big(1+\sigma_2+\sigma_3+\sigma_4+\sigma_5\big)\\
&= x_1^{m_1-1}x_2^{m_2-1}x_3^{m_3-1} + \big(x_{3,2}^{m_1-1}x_1^{m_2-1}x_2^{m_3-1}-x_{3,2}^{m_1-1}x_1^{m_2-1}x_3^{m_3-1}\big)\\
&+\big(x_1^{m_1-1}x_{3,2}^{m_2-1}x_2^{m_3-1} - x_1^{m_1-1}x_{3,2}^{m_2-1}x_3^{m_3-1}\big)\\
&+\big(x_{2,1}^{m_1-1}x_{3,1}^{m_2-1}x_1^{m_3-1} - x_{1,2}^{m_1-1}x_{3,2}^{m_2-1}x_2^{m_3-1}\big)\\
&+\big(x_{3,2}^{m_1-1}x_{3,1}^{m_2-1}x_3^{m_3-1}-x_{2,3}^{m_1-1}x_{2,1}^{m_2-1}x_2^{m_3-1}\big)\\
\end{align*}
are $e\tbinom{m_1,m_2}{n_1,n_2}\delta\tbinom{m_3}{n_3}$ and $h\tbinom{m_1,m_2,m_3}{n_1,n_2,n_3}$, respectively, where we set $x_{i,j}=x_i-x_j$.
Therefore, the identity \eqref{eq:c_deform} follows from \eqref{eq7_2}.
\end{proof}

We are ready to prove Theorem \ref{thm:rest_even_C^j}.

\begin{proof}[Proof of Theorem \ref{thm:rest_even_C^j}]
By Corollary \ref{7_2}, the matrix $C_k^{(j)}$ is written as 
\begin{equation*}\label{eq:decomp_Ckj} 
C_k^{(j)} = \left(e\tbinom{m_1,m_2}{n_1,n_2} \delta\tbinom{m_3}{n_3} \right)_{\begin{subarray}{c} (m_1,m_2,m_3)\in \mathbb{I}_k^{(3)} \\ (n_1,n_2,n_3)\in \mathbb{I}_k(\mathbf{aae})  \end{subarray}} \cdot \left(h\tbinom{\bf m}{\bf n} \right)_{\begin{subarray}{l} {\bf m} \in \mathbb{I}_k(\mathbf{aae})  \\ {\bf n}\in \mathbb{I}_k^{(j)} \end{subarray}}. 
\end{equation*} 
The first matrix on the right can be written in terms of a block matrix with blocks $C_{k-n}$ for $n=2,4,\ldots,k-6$:
\begin{equation*}
\left(e\tbinom{m_1,m_2}{n_1,n_2} \delta\tbinom{m_3}{n_3} \right)_{\begin{subarray}{c} (m_1,m_2,m_3)\in \mathbb{I}_k^{(3)} \\ (n_1,n_2,n_3)\in \mathbb{I}_k (\mathbf{aae}) \end{subarray}}  = {\rm diag} ( C_{k-2},C_{k-4},\ldots, C_6 ),
\end{equation*}
where for $k$ even the matrix $C_k$ is defined to be
\begin{equation*}\label{matC}
C_k= \left( e\tbinom{\bf m}{\bf n} \right)_{\begin{subarray}{l} {\bf m}\in  \mathbb{I}_k (\mathbf{oo})\\ {\bf n}\in \mathbb{I}_k (\mathbf{aa}) \end{subarray}},
\end{equation*}
whose rows and columns are indexed by ${\bf m}$ and ${\bf n}$ in the sets $\mathbb{I}_k (\mathbf{oo})$ and $\mathbb{I}_k (\mathbf{aa})$, respectively.
Let $\ker C_k$ be the $\Q$-vector space of left annihilators of $C_k$.
There is an embedding $\bigoplus_{1<n<k:{\rm even}} \ker C_{k-n} \rightarrow \ker C_k^{(j)}$ that sends $( a_{n_1,n_2})_{(n_1,n_2)\in \mathbb{I}_{k-n}(\mathbf{oo})} \in  \ker C_{k-n}  $ to $( a_{n_1,n_2}\delta\tbinom{n}{n_3})_{(n_1,n_2,n_3)\in \mathbb{I}_k(\mathbf{ooe})} \in  \ker C_k^{(j)} $.
Hence, Theorem \ref{thm:rest_even_C^j} follows from the fact that the map
\begin{align*}
W_{k-n}^{+,0} & \longrightarrow \ker C_{k-n}\\
\sum_{(n_1,n_2)\in \mathbb{I}_{k-n}(\mathbf{oo})} a_{n_1,n_2}x_1^{n_1-1}x_2^{n_2-1} & \longmapsto ( a_{n_1,n_2})_{(n_1,n_2)\in \mathbb{I}_{k-n}(\mathbf{oo})} 
\end{align*}
is an isomorphism as $\Q$-vector spaces, which is \cite[Proposition 3.4]{T} and equivalent to the result of Baumard and Schneps \cite[Proposition 3.2]{BS}. 
We complete the proof.
\end{proof}

\begin{remark}
As a consequence of Corollary \ref{7_2} and Theorem \ref{thm:rankC}, one can give an explicit formula for the parity theorem of $\zeta^\m(n_1,n_2,n_3)$ with $(n_1,n_2,n_3)\in \mathbb{I}_k^{(j)}$ (the depth 2 case is mentioned in \eqref{eq:dep2_parity}).
A similar result can be found in \cite[Eq.~(1.12)]{P} (but not known that the formula is lifted to motivic multiple zeta values), and the formula may be different from the formula below due to linear relations among multiple zeta values.

\begin{theorem}\label{9_1}
Let $k$ be a positive even integer. 
For each $(n_1,n_2,n_3)\in\mathbb{I}_k^{(j)}$, we have
\[ \zeta^\m(n_1,n_2,n_3) \equiv \sum_{(k_1,k_2,k_3)\in \mathbb{I}_k(\mathbf{aae})}h\tbinom{k_1,k_2,k_3}{n_1,n_2,n_3} \zeta^\m(k_1,k_2) \zeta^\m(k_3) \mod \mathfrak{D}_2\mathcal{H}_k,\]
where $h\tbinom{k_1,k_2,k_3}{n_1,n_2,n_3} \in\Z$ is defined in \eqref{eq7_3}.
\end{theorem}
The proof is very similar to the proof of Theorem \ref{thm:rankC}, so is omitted.

\end{remark}

\subsection{Proof of Theorem \ref{thm:B^3}}
In this subsection, we prove Theorem \ref{thm:B^3}.

\begin{proof}[Proof of Theorem \ref{thm:B^3}]
By definition, the matrix $B_k^{(3)}$ can be written in terms of a block diagonal matrix with blocks $B_{k-n}$ for $n=3,5,\ldots,k-5$:
\begin{equation*}\label{eq:decomp_Bk3}
\begin{aligned}
B_k^{(3)}&={\rm diag}(B_{k-3},B_{k-5},\ldots,B_5),
\end{aligned}
\end{equation*}
where the matrix $B_k$ is defined for $k$ odd by 
\begin{equation*}
B_k= \left( e\tbinom{m_1,m_2}{n_1,n_2} \right)_{\begin{subarray}{l} (m_1,m_2)\in \mathbb{I}_{k}(\mathbf{oe})\\ (n_1,n_2)\in\mathbb{I}_{k}(\mathbf{oe}) \end{subarray}}.
\end{equation*}
Since there is an embedding $\bigoplus_{3\le n \le k-5:{\rm odd}} \ker B_{k-n}\rightarrow \ker B_k^{(3)}$, Theorem \ref{thm:B^3} follows from a result by Zagier on the matrix $\mathcal{B}_K$ in \cite[\S6]{Z3}.
Since the matrix $\mathcal{B}_K$ is slightly different from our matrix $B_k$, we sketch the proof.

It is easily seen that the assertion that a row vector $(a_{n_1,n_2})_{(n_1,n_2) \in \mathbb{I}_k(\mathbf{oe})}$ lies in $\ker B_k$ is equivalent to the statement that the polynomial $q(x_1,x_2)=\sum_{(n_1,n_2)\in \mathbb{I}_k(\mathbf{oe})} a_{n_1,n_2} x_1^{n_1-1}x_2^{n_2-1}$ satisfies
\[ q(x_1,x_2)-q(x_2-x_1,x_2)+q(x_2-x_1,x_1) = \mbox{(odd polynomial in $x_1$)}.\]
With this, one finds that the map
\begin{align*}
W_{k+1}^{+,0} & \longrightarrow \ker B_k\\
p(x_1,x_2)=\sum_{(n_1,n_2)\in \mathbb{I}_{k+1}(\mathbf{oo})}a_{n_1,n_2} x_1^{n_1-1}x_2^{n_2-1} &\longmapsto (n_2 a_{n_2+1,n_1})_{(n_1,n_2)\in \mathbb{I}_{k}(\mathbf{oe})},
\end{align*}
where the image is the coefficient vector of $\frac{dp}{dx_1}(x_2,x_1)$, and the map
\begin{align*}
W_{k-1}^{-} & \longrightarrow \ker B_k\\
p(x_1,x_2)=\sum_{(n_1,n_2)\in \mathbb{I}_{k-1}(\mathbf{ee})}a_{n_1,n_2} x_1^{n_1-1}x_2^{n_2-1} &\longmapsto (a_{n_1-1,n_2})_{(n_1,n_2)\in \mathbb{I}_{k}(\mathbf{oe})},
\end{align*}
where the image is the coefficient vector of $x_1p(x_1,x_2)$, are well-defined.
By definition, the injectivity of each of these maps is obvious.
The injectivity of the combined map $W_{k+1}^{+,0}\oplus W_{k-1}^- \rightarrow \ker B_k$ is also obvious since the images have incompatible symmetry properties. 
We complete the proof.
\end{proof}

For example, the matrices
\[ B_{11}=\left(
\begin{array}{cccc}
 0 & 0 & 0 & -2 \\
 -6 & 0 & -4 & -4 \\
 -15 & -21 & -20 & -6 \\
 -36 & -126 & -84 & -8 \\
\end{array}
\right),\ 
B_{13}=\left(
\begin{array}{ccccc}
 0 & 0 & 0 & 0 & -2 \\
 -6 & 0 & 0 & -4 & -4 \\
 -15 & -15 & -6 & -20 & -6 \\
 -28 & -78 & -84 & -56 & -8 \\
 -55 & -330 & -462 & -165 & -10 \\
\end{array}
\right)\]
have left annihilators 
\[\left(
\begin{array}{cccc}
 -4 & 9 & -6 & 1 \\
\end{array}
\right),\ \left(
\begin{array}{ccccc}
 4 & -25 & 42 & -25 & 4 \\
\end{array}
\right),\]
respectively.
These examples can be found in \cite[p.995]{Z3}.
The equality $\dim_\Q \ker B_k \stackrel{?}{=} \dim_\Q W_{k-1}^- + \dim_\Q W_{k+1}^{+,0}$ is not known.
We also expect
\[ \sum_{N>0} \dim_{\Q}\ker B_k^{(3)} x^k \stackrel{?}{=} \big(x+ \frac{1}{x}\big) \odd(x)\cusp(x),\]
which has been checked by Mathematica up to $k=40$.

\subsection{Proof of Theorem \ref{thm:B^3+E^3}}

We first give a precise statement of Theorem \ref{thm:B^3+E^3} and then prove it.

For $k$ even, consider an extended index set of $ \mathbb{I}_k^{(3)}$ allowing the cases when $n_3=0$:
\[ \widehat{\mathbb{I}}_k^{(3)} =\{ {\bf n}= (n_1,n_2,n_3)\in \Z^3_{\ge0} \mid k=n_1+n_2+n_3,\ n_1,n_2\in\mathbf{o},\ n_3\ge0:{\rm even} \}.\]
We set
\begin{align*}
\V_k^{(3)} = \{ (a_{\bf n})_{{\bf n}\in \mathbb{I}_k^{(3)}} \mid a_{\bf n}\in\Q \}, \ \widehat{\V}_k^{(3)} = \{ (a_{\bf n})_{{\bf n}\in \widehat{\mathbb{I}}_k^{(3)}} \mid a_{\bf n}\in\Q \}.
\end{align*}
Our target space ${\rm Im}\, B_k^{(3)} \cap \ker E_k^{(3)}$, which is viewed as a subspace of $\V_k^{(3)}$, can be embedded into the space $\widehat{\V}_k^{(3)}$ via embedding
\begin{equation*}
\begin{aligned}
i_0 : \V_k^{(3)} &\longrightarrow \widehat{\V}_k^{(3)}\\
(a_{\bf n})_{{\bf n}\in \mathbb{I}_k^{(3)}} &\longmapsto (a_{\bf n})_{{\bf n}\in \widehat{\mathbb{I}}_k^{(3)}},
\end{aligned}
\end{equation*}
where we simply put $a_{2n+1,k-2n-1,0}=0$ for all $1\le n \le k/2-2$.

Let $\widehat{V}_k$ be the $|\widehat{\mathbb{I}}_k^{(3)}|$-dimensional $\Q$-vector space spanned by the set $\{x_1^{n_1-1}x_2^{n_2-1} x_3^{n_3-1}\mid (n_1,n_2,n_3)\in \widehat{\mathbb{I}}_k^{(3)}\}$, which is a subspace of $\Q[x_1,x_2,x_3^{\pm 1}]$ and isomorphic to the space $\widehat{\V}_k^{(3)} $.
The isomorphism is denoted by $\rho(=\rho^{(k)})$:
\begin{equation*}
\begin{aligned}
\rho:\widehat{V}_k &\longrightarrow \widehat{\V}_k^{(3)}\\
  x_1^{m_1-1}x_2^{m_2-1} x_3^{m_3-1}&\longmapsto \big(\delta\tbinom{m_1,m_2,m_3}{n_1,n_2,n_3}\big)_{(n_1,n_2,n_3)\in \widehat{\mathbb{I}}_k^{(3)}}.
\end{aligned}
\end{equation*}
We note that the space $\widehat{P}_k^+$ is a subspace of $\widehat{V}_k$, and hence the space $\rho(\widehat{P}_k^+)$ is a subspace of $\widehat{\V}_k^{(3)}$.
With these notations, the precise statement of Theorem \ref{thm:B^3+E^3} is as follows.
We note that for ${\bf m},{\bf n}\in  \widehat{\mathbb{I}}_k^{(3)}$ the integer $e\tbinom{\bf m}{\bf n}$ is well-defined (see \eqref{eq:def_e3}).

\begin{theorem}\label{8_4}
Let $k$ be a positive even integer. 
Define a square matrix $L_k$ by
\[L_k=\left( e\tbinom{\bf m}{\bf n}-\delta\tbinom{\bf m}{\bf n} \right)_{\begin{subarray}{c} {\bf m}\in \widehat{\mathbb{I}}_k^{(3)}\\ {\bf n}\in \widehat{\mathbb{I}}_k^{(3)}\end{subarray}},\]
which is viewed as a linear map $L_k:\widehat{\V}_k^{(3)}\rightarrow \widehat{\V}_k^{(3)}$ given by $L_k(v)=v\cdot L_k$ for $v\in \widehat{\V}_k^{(3)}$.
Then, the map
\[ L_k : \rho(\widehat{{P}}_k^+) \longrightarrow i_0\big({\rm Im}\, B_k^{(3)} \cap \ker E_k^{(3)}\big)\]
is well-defied.
\end{theorem}

\begin{proof}
The proof is done by showing the following claims:\\
(Claim 1) $L_k\big( \rho\big(\widehat{{P}}_k^+\big) \big)\subset i_0 \big( \ker E_k^{(3)} \big)$,\\
(Claim 2) $L_k\big( \rho\big(\widehat{{P}}_k^+\big) \big)\subset i_0 \big( {\rm Im}\, B_k^{(3)} \big)$,\\
from which, by $i_0\big({\rm Im}\, B_k^{(3)} \cap \ker E_k^{(3)}\big)=i_0\big({\rm Im}\, B_k^{(3)}\big) \cap i_0\big( \ker E_k^{(3)}\big)$, Theorem \ref{8_4} follows.

(Claim 1).
Define a square matrix $\widehat{E}_k^{(3)}$ by
\[
\widehat{E}_k^{(3)} = \left( e\tbinom{\bf m}{\bf n} \right)_{\begin{subarray}{c} {\bf m}\in \widehat{\mathbb{I}}_k^{(3)}\\ {\bf n}\in \widehat{\mathbb{I}}_k^{(3)} \end{subarray}}.\]
We first prove $L_k (\rho\big(\widehat{{P}}_k^+\big) ) \subset \ker \widehat{E}_k^{(3)}$, and then $L_k\big( \rho\big(\widehat{{P}}_k^+\big) \big)\subset i_0 \big( \ker E_k^{(3)} \big)$.

For any $p(x_1,x_2,x_3)\in \widehat{{V}}_k$, one easily sees that for each $(n_1,n_2,n_3)\in \widehat{\mathbb{I}}_k^{(3)}$
\begin{equation}\label{eq:cof1}
\begin{aligned}
&\mbox{the coefficient of $x_1^{n_1-1}x_2^{n_2-1} x_3^{n_3-1}$ in $\big(p\big| (\sigma_1+\sigma_2)\big) \big|(1+\sigma_1+\sigma_2)$}\\
&=\mbox{$(n_1,n_2,n_3)$-th entry of the row vector $\rho(p)\cdot L_k\cdot \widehat{E}_k^{(3)}$},
\end{aligned}
\end{equation}
where  $\sigma_i$'s are defined in \eqref{eq7_1}.
Since $p(x_1,x_2,x_3)\in \widehat{{P}}_k^+$ satisfies 
\begin{equation}\label{eq8_1}
p(x_1,x_2,x_3)\big| (1+\sigma_1)=0
\end{equation}
(i.e. $p\big| \sigma_1=-p$), we have 
\begin{equation}\label{eq:mid1} 
\big(p\big| (\sigma_1+\sigma_2)\big) \big|(1+\sigma_1+\sigma_2) = \big(p\big| \sigma_2\big) \big|\sigma_1+\big(p\big| \sigma_2\big) \big|\sigma_2.
\end{equation}
The antisymmetry of the coefficients of even period polynomials shows
\begin{equation}\label{eq8_2} 
p(x_1,x_2,x_3)+p(x_2,x_1,x_3)=0.
\end{equation}
Using \eqref{eq8_1} and \eqref{eq8_2}, one can easily see that the right-hand side of \eqref{eq:mid1} is reduced to 0 (we refer the reader to \cite[Eq.~(3.14)]{T} for the detailed verification).
Hence 
\begin{equation}\label{eq8_3}
\big(p\big| (\sigma_1+\sigma_2)\big) \big|(1+\sigma_1+\sigma_2)=0
\end{equation}
holds for $p\in \widehat{P}_k^+$, and by \eqref{eq:cof1} we have $L_k(\rho(\widehat{{P}}_k^+))\subset \ker\widehat{E}_k^{(3)}$.

Let us turn to the proof of the inclusion $L_k(\rho(\widehat{{P}}_k^+))\subset i_0(\ker E_k^{(3)})$.
We fix $p\in \widehat{{P}}_k^+$ and write $\rho(p)=(a_{\bf n})_{{\bf n}\in \widehat{\mathbb{I}}^{(3)}_k}$ and $L_k(  \rho(p)) = \big( c_{\bf n} \big)_{{\bf n}\in \widehat{\mathbb{I}}_k^{(3)}}$.
Since $e\tbinom{m_1,m_2,m_3}{n_1,n_2,n_3}=0$ whenever $m_3>0$ and $n_3=0$ (see \eqref{eq:def_e3}), one computes for $n_1,n_2\ge3$ odd with $n_1+n_2=k$
\begin{align*}
 c_{n_1,n_2,0}&=\sum_{\substack{m_1+m_2=k\\m_1,m_2\ge3:{\rm odd}}} a_{m_1,m_2,0} \bigg( e\tbinom{m_1,m_2,0}{n_1,n_2,0} - \delta\tbinom{m_1,m_2}{n_1,n_2} \bigg)\\
 & = \sum_{\substack{m_1+m_2=k\\m_1,m_2\ge3:{\rm odd}}} a_{m_1,m_2,0} \bigg(b_{n_1,n_2}^{m_1}- \delta\tbinom{m_2,m_1}{n_1,n_2} \bigg)\\
 & =  \sum_{\substack{m_1+m_2=k\\m_1,m_2\ge3:{\rm odd}}} a_{m_1,m_2,0} \bigg(b_{n_1,n_2}^{m_1}+ \delta\tbinom{m_1,m_2}{n_1,n_2} \bigg)=0,
\end{align*}
where for the third equality we have used \eqref{eq8_2}
and the last equality is obtained from \eqref{eq8_1}.
We have shown $L_k(\rho(p))\in \ker \widehat{E}_k^{(3)}$, so we get for any ${\bf n}\in \widehat{\mathbb{I}}_k^{(3)}$
\[  \sum_{{\bf m}\in \widehat{\mathbb{I}}_k^{(3)} } c_{\bf m} e\tbinom{\bf m}{\bf n}=0.\]
Since $ c_{n_1,n_2,0}=0$, this leads for any ${\bf n} \in \mathbb{I}_k^{(3)}$ to the relation
\[\sum_{{\bf m}\in \mathbb{I}_k^{(3)} } c_{\bf m } e\tbinom{\bf m}{\bf n}=0,\]
which implies that the row vector $(c_{\bf n})_{{\bf n}\in \mathbb{I}_k^{(3)}}$ lies in $\ker E_k^{(3)}$.
The claim 1 is done.

(Claim 2).
Let us prove $L_k(\rho(\widehat{{P}}_k^+))\subset i_0(\ker B_k^{(3)})$.
For $k$ even, define a square matrix $\widehat{B}_k^{(3)}$ by
\[
\widehat{B}_k^{(3)} = \left( \delta\tbinom{m_1}{n_1}e\tbinom{m_2,m_3}{n_2,n_3} \right)_{\begin{subarray}{c} (m_1,m_2,m_3)\in \widehat{\mathbb{I}}_k^{(3)}\\ (n_1,n_2,n_3)\in \widehat{\mathbb{I}}_k^{(3)} \end{subarray}}.\]
We first show $L_k(\rho(\widehat{{P}}_k^+))=\widehat{B}_k^{(3)}(\rho(\widehat{{P}}_k^+))$, and then $\widehat{B}_k^{(3)}(\rho(\widehat{{P}}_k^+)) \subset i_0 \big( {\rm Im}\, B_k^{(3)} \big)$.

Again, we fix $p\in \widehat{{P}}_k^+$ and write $\rho(p)=(a_{\bf n})_{{\bf n}\in \widehat{\mathbb{I}}^{(3)}_k}$ and $L_k(  \rho(p)) = \big( c_{\bf n} \big)_{{\bf n}\in \widehat{\mathbb{I}}_k^{(3)}}$.
From \eqref{eq:def_e3}, the $c_{\bf n}$ can be computed as follows:
\begin{align*} 
c_{n_1,n_2,n_3} &=  \sum_{(m_1,m_2,m_3)\in \widehat{\mathbb{I}}_k^{(3)}} a_{m_1,m_2,m_3} \bigg( \delta\tbinom{m_3}{n_3} b_{n_1,n_2}^{m_1} + \delta\tbinom{m_2}{n_1} b_{n_2,n_3}^{m_1} \bigg) \\
&= \sum_{(m_1,m_2,m_3)\in \widehat{\mathbb{I}}_k^{(3)}} a_{m_1,m_2,m_3} \bigg(- \delta\tbinom{m_1,m_2,m_3}{n_1,n_2,n_3} - \delta\tbinom{m_1}{n_1} b_{n_2,n_3}^{m_2} \bigg),
\end{align*}
where we have used the relations \eqref{eq8_1} and \eqref{eq8_2} for the last equality.
It follows that the above last term coincides with the $(n_1,n_2,n_3)$-th entry of the row vector $-\widehat{B}_k^{(3)}(  \rho(p) )$, i.e. we have $L_k(\rho(p))=-\widehat{B}_k^{(3)}(  \rho(p) )$, and hence $L_k(\rho(\widehat{{P}}_k^+))=\widehat{B}_k^{(3)}(  \rho(\widehat{{P}}_k^+) )$.

To prove $\widehat{B}_k^{(3)}(\rho(\widehat{{P}}_k^+)) \subset i_0 \big( {\rm Im}\, B_k^{(3)} \big)$, we need the following lemma.
For $k$ odd, let
\[ \widehat{\mathbb{I}}_k(\mathbf{o}\mathbf{e}) = \mathbb{I}_k(\mathbf{o}\mathbf{e}) \cup \{(k,0)\}\]
and define an extended matrix $\widehat{B}_k$ of $B_k$ by
\[
\widehat{B}_k=\left(  e\tbinom{\bf m}{\bf n} \right)_{\begin{subarray}{c} {\bf m}\in \widehat{\mathbb{I}}_k(\mathbf{o}\mathbf{e}) \\ {\bf n} \in \widehat{\mathbb{I}}_k(\mathbf{o}\mathbf{e}) \end{subarray}}.
\]

\begin{lemma}\label{8_5}
For each odd integer $k\ge3$, there exists an element $(a_{\bf n})_{{\bf n}\in \widehat{\mathbb{I}}_k(\mathbf{o}\mathbf{e})  } \in \ker \widehat{B}_k$ such that $a_{k,0}\neq 0$.
\end{lemma}
\begin{proof}
Since $\widehat{B}_3=\left(e\tbinom{3,0}{3,0} \right)=(0)$, we have $\ker \widehat{B}_3=\Q$.
For the case $k\ge5$ odd, such element in the statement is obtained from the extended period polynomial corresponding to the Eisenstein series, which was obtained in \cite{Z4}.
For $k\ge4$ even, let
\begin{equation}\label{eq:period_poly_Eis}
\widehat{G}_k (x_1,x_2) := 4\sum_{\substack{n_1+n_2=k\\n_1,n_2\ge0}} \beta_{n_1}\beta_{n_2} x_1^{n_1-1}x_2^{n_2-1} \in \frac{1}{x_1}\Q[x_1,x_2]+\frac{1}{x_2}\Q[x_1,x_2],
\end{equation}
where $\beta_k$ is defined in \eqref{beta} (note that the Laurent polynomial $\widehat{G}_k (x_1,x_2)$ corresponds to \cite[Eq.~(11)]{Z4}).
It follows from \cite[Proposition in p.453]{Z4} that $\widehat{G}_k (x_1,x_2) - \widehat{G}_k (x_1+x_2,x_2) - \widehat{G}_k (x_1+x_2,x_1) =0$. 
Letting $x_1\rightarrow -x_1$ and using $\widehat{G}_k (-x_1,x_2) =\widehat{G}_k (x_1,-x_2)=-\widehat{G}_k (x_1,x_2)  $, we have
\begin{equation}\label{eq:rel_G} 
\widehat{G}_k (x_1,x_2) +\widehat{G}_k (x_2-x_1,x_2) -\widehat{G}_k (x_2-x_1,x_1) =0.
\end{equation}
For $k\ge5$ odd, we set $p(x_1,x_2) = x_1\widehat{G}_{k-1} (x_1,x_2)$ and define rational numbers $a_{n_1,n_2}$'s by
\begin{equation*}
p(x_1,x_2)-p(0,x_2) = \sum_{\substack{n_1+n_2=k\\n_1\ge2\\n_2\ge0}} a_{n_1,n_2} x_1^{n_1-1}x_2^{n_2-1},
\end{equation*}
i.e. $a_{n_1,n_2}=4\beta_{n_1-1}\beta_{n_2}$ if $(n_1,n_2)\in \widehat{\mathbb{I}}_k(\mathbf{o}\mathbf{e})$ and $a_{n_1,n_2}=0$ otherwise.
Since $a_{k,0} \neq 0$, the proof is done by showing that the row vector $(a_{\bf n})_{{\bf n}\in \widehat{\mathbb{I}}_k(\mathbf{oe})}$ lies in $\ker \widehat{B}_k$.

It can be shown that the assertion that a row vector $(b_{\bf n})_{{\bf n} \in \widehat{\mathbb{I}}_k(\mathbf{oe})}$ lies in $\ker \widehat{B}_k$ is equivalent to the statement that the polynomial $q(x_1,x_2)=\sum_{(n_1,n_2)\in \widehat{\mathbb{I}}_k(\mathbf{oe})} b_{n_1,n_2} x_1^{n_1-1}x_2^{n_2-1}$ satisfies
\[ q(x_1,x_2)-q(x_2-x_1,x_2)+q(x_2-x_1,x_1) = \mbox{(odd polynomial in $x_1$)}.\]
Now let $q(x_1,x_2)=p(x_1,x_2)-p(0,x_2)$ and computes
\begin{align*}
&q(x_1,x_2)-q(x_2-x_1,x_2)+q(x_2-x_1,x_1) \\
&= p(x_1,x_2) - p(x_2-x_1,x_2) + p(x_2-x_1,x_1)-p(0,x_1)\\
&= x_2 \widehat{G}_{k-1}(x_1,x_2) - p(0,x_1)= q(x_2,x_1),
\end{align*}
where for the second equality we have used \eqref{eq:rel_G} and the last equality follows from $\widehat{G}_k(x_1,x_2)=\widehat{G}_k(x_2,x_1)$.
Since $q(x_2,-x_1)=\sum_{(n_1,n_2)\in \widehat{\mathbb{I}}_k(\mathbf{oe})} a_{n_1,n_2}x_1^{n_1-1}(-x_2)^{n_2-1}=  -q(x_2,x_1)$, the proof is concluded.
\end{proof}

Let us turn to the proof of the inclusion $\widehat{B}_k^{(3)} \big(\rho\big( \widehat{{P}}_k^+ \big) \big) \subset i_0 \big( {\rm Im}\, B_k^{(3)} \big)$ for $k$ even.
It suffices to show that for any $v\in \rho\big(\widehat{{P}}_k^+\big)$, there exists $v'\in \V^{(3)}_k$ such that 
\[
\widehat{B}_k^{(3)}(v)=i_0\big(B_k^{(3)}(v')\big).
\]
We note that the matrix $\widehat{B}_k^{(3)}$ can be written in terms of a block diagonal matrix with blocks $\widehat{B}_{k-n}$ for $n=3,5,\ldots,k-3$:
\begin{align}\label{eq:decomp_hatBk3}
\widehat{B}_k^{(3)} &= {\rm diag}(\widehat{B}_{k-3},\widehat{B}_{k-5},\ldots,\widehat{B}_3).
\end{align}
For each $\widehat{B}_{k-n}$, Lemma \ref{8_5} shows the existence of a row vector $(b_{\bf n})_{{\bf n} \in \widehat{\mathbb{I}}_{k-n}(\mathbf{o}\mathbf{e})} \in \ker \widehat{B}_{k-n}$ such that $b_{k-n,0}=1$.
By \eqref{eq:decomp_hatBk3} we see that the row vector $(\delta\tbinom{n}{n_1} b_{n_2,n_3})_{(n_1,n_2,n_3) \in \widehat{\mathbb{I}}_{k}(\mathbf{o}\mathbf{o}\mathbf{e})}$ lies in $\ker \widehat{B}_k^{(3)}$.
This says that for each $(m_1,m_2)\in \mathbb{I}_k(\mathbf{o}\mathbf{o})$, there is a row vector $v_{m_1,m_2} \in \ker \widehat{B}_k^{(3)}$ whose $(n_1,n_2,0)$-th entry is $\delta\tbinom{m_1,m_2}{n_1,n_2} $ for all $(n_1,n_2,0)\in \widehat{\mathbb{I}}_k^{(3)}$. 
Then, for any $\rho(p)=(a_{\bf n})_{{\bf n}\in \widehat{\mathbb{I}}_k^{(3)}} \in \rho(\widehat{{P}}_k^+)$, the element
\[  \rho(p) - \sum_{(n_1,n_2)\in   \mathbb{I}_k(\mathbf{o}\mathbf{o})} a_{n_1,n_2,0}\cdot v_{n_1,n_2}\]
lies in $i_0\big(\V_k^{(3)}\big)$.
Therefore, there is an element $v' \in \V_k^{(3)}$ satisfying $\widehat{B}_k^{(3)} (\rho(p) )=\widehat{B}_k^{(3)} (i_0(v' ))$.
Since $e\tbinom{m,k-m}{k,0}=0$ whenever $k-m>0$, it is easily seen that the equality $\widehat{B}_k^{(3)} (i_0(v' ))=i_0 \big( B_k^{(3)} (v') \big)$ holds.
Hence $\widehat{B}_k^{(3)} \big(\rho\big( \widehat{{P}}_k^+ \big) \big)$ belongs to $ i_0 \big( {\rm Im}\, B_k^{(3)} \big)$, so does $L_k \big(\rho\big( \widehat{{P}}_k^+ \big) \big) $.
\end{proof}

We expect that the map $L_k$ in Theorem \ref{8_4} is bijective, because our numerical experiments show that
\[ {\rm Im}\, B_k^{(3)} \cap \ker E_k^{(3)} \stackrel{?}{=} \ker E_k^{(3)}\]
and 
\[ \sum_{k>0} \dim_\Q   \ker E_k^{(3)} x^k  \stackrel{?}{=} \frac{1}{x^2}\even(x) \cusp (x),\]
which have both been verified by using Mathematica up to $k=40$.

Let us give an example of Theorem \ref{8_4}.
One finds that all left annihilators of the matrix
\[
C_{12}^{(3)}=\left(
\begin{array}{cccccc}
 0 & 0 & 0 & 0 & 0 & 4 \\
 0 & 0 & 12 & 0 & 8 & 8 \\
 0 & 0 & 42 & 0 & 70 & 12 \\
 0 & 0 & 12 & 20 & 8 & -12 \\
 60 & 100 & 64 & -20 & 16 & -24 \\
 42 & 70 & 42 & 0 & 0 & -30 \\
\end{array}
\right)
\]
are a rational multiple of the row vector
\[ 
\left(
\begin{array}{cccccc}
 20 & 14 & 20 & -63 & -63 & 90 \\
\end{array}
\right).\]
Theorem \ref{8_4} characterizes this row vector: it is obtained from $L_{12}(\rho(\widehat{P}_{12}^+))$.
We remark that the right annihilator of $C_{12}^{(3)}$ gives a simple relation of the form
\[3 \zeta^\m(3,5,4)-5\zeta^\m(3,3,6) \equiv 0 \mod \mathfrak{D}_2\mathcal{H}_{12}.\]

\subsection{Remark on $C_k^{(1)}$}

There have to be more elements in $\ker C_k^{(1)}$ corresponding to the $x^k$-term in 
$$\frac{1}{x^2}\even(x)\cusp(x)-\even(x)\cusp(x)=\cusp(x),$$
which is nothing but the dimension of $S_k$.
For this, the first author observed a remarkable connection with the ``derivative" of odd period polynomials removing the first and last terms (which should be called the {\it restricted odd period polynomial}).
For example, the row vector $(35,-42,15,0,0,0)$ is a basis of $\ker C_{12}^{(1)}$, and these coefficients are obtained from the derivative of the odd period polynomial \eqref{eq:odd-period-poly-wt12} ignoring the term $4x^9y+4xy^9$.
In general, for $p(x_1,x_2) \in W_k^{-}$, let $a_{n_1,n_2,n_3}$ be the coefficient of $x_1^{n_1-1}x_2^{n_2-1}x_3^{n_3-1}$ in 
\[ \frac{x_3}{x_2} \frac{dp}{dx_1} (x_1,x_2).\]
Then, the row vector $(a_{\bf n})_{{\bf n}\in \mathbb{I}_k^{(3)}}$ seemingly lies in $\ker C_k^{(1)}$.
This has been verified by Mathematica up to $k=40$.

\section{Further remarks}

Denote by $\depth$ the depth filtration  (see \cite[\S4]{B}), which is an increasing filtration on $\Zsp$:
\[ \depth_0\Zsp=\Q\subset \depth_1\Zsp\subset \cdots \subset \depth_r \Zsp \subset  \cdots \subset \Zsp.\] 
We note that the space $\Zsp$ has the structure of a filtered algebra with respect to the depth filtration $\depth$, i.e. 
\[ \depth_r \Zsp \cdot \depth_s \Zsp \subset \depth_{r+s} \Zsp\]
holds for any $r,s\ge0$.

Consider the following graded $\Q$-algebra:
\[ {\rm gr}^\depth \Zsp := \bigoplus_{r\ge1} \depth_r \Zsp \big/ \depth_{r-1}\Zsp. \]
We denote by 
\[\zeta^\m_\depth(n_1,\ldots,n_r)\]
 an image of $\zeta^\m(n_1,\ldots,n_r)$ in the bigraded $\Q$-algebra ${\rm gr}^\depth \Zsp$, called the {\it depth-graded motivic multiple zeta value}.
Note that by definition $\zeta^\m_\depth(2)$ is non-zero.
Now we define the $j$-th almost totally odd motivic multiple zeta value.

\begin{definition}\label{6_1}
For an integer $r\ge j\ge 1$, the depth-graded motivic multiple zeta value $\zeta^\m_\depth(n_1,\ldots,n_r)$ is called the $j$-th almost totally odd motivic multiple zeta value if $n_j\ge2$ is even and other $n_i$'s are odd and greater than 1.
\end{definition}

The reminder of this section is devoted to illustrating some expectation for the almost totally odd motivic multiple zeta values.
As a prototype, we have in mind an analogues story for the study of totally odd motivic multiple zeta values developed by Brown \cite[\S10]{B}.
They are elements $\zeta^\m(2n_1+1,2n_2+1,\ldots,2n_r+1)$'s ($n_i\ge1$) in the space ${\rm gr}^\depth \mathcal{A}$.
In \cite[Conjecture 4]{B}, Brown recast the Broadhurst-Kreimer conjecture \cite{BK} as a statement of the homology of the depth-graded motivic Lie algebra $\mathfrak{d}$ (see \cite[\S2.5]{B2} and \cite[\S4]{B}).
It is believed that the $\Q$-algebra generated by totally odd motivic multiple zeta values is isomorphic to the graded dual of the universal enveloping algebra of the Lie subalgebra $\mathfrak{g}^{\rm odd}$ of $\mathfrak{d}$ generated by canonical generators $\sigma_{2i+1}^{(1)}$ in depth 1.
The generators $\sigma_{2i+1}^{(1)}$ are also believed to be subject only to the quadratic relations obtained from restricted even period polynomials (the Ihara-Takao relation).
These lead to the uneven part of the Broadhurst-Kreimer conjecture \cite[Conjecture 5]{B} stating that the generating series of the dimension of the $\Q$-vector space spanned by all totally odd multiple zeta values of weight $k$ and depth $r$ is given by
\[ \frac{1}{1-\odd(x)y+\cusp(x)y^2} .\]

Now consider the $\Q$-vector space $\Zsp^{al}_{k,r}$ spanned by $\zeta^\m_{\depth}(2n_1)\zeta^\m_{\depth}(2n_2+1,\ldots,2n_{r}+1)$ ($n_i\ge1$) of weight $k$ and depth $r$ for positive integers $k>r>1$.
We let $\Zsp^{al}_{k,1}=\Q\zeta_\depth^\m(k)$ if $k$ is even and $\Zsp^{al}_{k,1}=\{0\}$ if $k$ is odd.
According to the uneven part of the Broadhurst-Kreimer conjecture, it is naturally predicted that 
\begin{align}
\label{eq:conj-dim-al}
\sum_{k,r>0} \dim\Zsp^{al}_{k,r} x^ky^r&\stackrel{?}{=} \frac{\even(x)y}{1-\odd(x)y+\cusp(x)y^2} \\
\notag &=\even(x)y + \even(x)\odd(x) y^2 + \big(\even(x)\odd(x)^2- \even(x)\cusp(x) \big)y^3+\cdots.
\end{align}
Indeed, the equality of the coefficients of $y^r$ in \eqref{eq:conj-dim-al} holds for $r=1,2,3$ (the case $r=3$ is due to Goncharov \cite[Theorem 2.6]{G1}).

Denote by $\Zsp^{al,(j)}_{k,r}$ the $\Q$-vector space spanned by all $j$-th almost totally odd motivic multiple zeta values of weight $k$ and depth $r$.
By the parity theorem we have $ \Zsp^{al,(j)}_{k,r}\subset  \Zsp^{al}_{k,r}$ for $r=1,2,3$ and $1\le j\le r$, which is not known for $r\ge4$.
Moreover, it follows from \cite[Theorem 2]{Z3} (see also \cite{M3}) that the equality $\Zsp^{al,(1)}_{k,2}= \Zsp^{al}_{k,2}$ holds.
According to Conjecture \ref{conj:dim_C^j} and \eqref{eq:conj-dim-al}, it is expected that the equality $ \Zsp^{al,(2)}_{k,3}\stackrel{?}{=} \Zsp^{al}_{k,3}$ holds.
In general, we are expecting the following conjecture.

\begin{conjecture}
For $k>r\ge3$, we have $\Zsp_{k,r}^{al,(r-1)}\stackrel{?}{=}\Zsp_{k,r}^{al}$.
\end{conjecture}

As a final remark, consider for $k$ even
\[ C_k^{(\mathbf{eee})} =  \left( c\tbinom{m_1,m_2,m_3}{n_1,n_2,n_3} \right)_{\begin{subarray}{c} (m_1,m_2,m_3)\in \mathbb{I}_k^{(3)} \\ (n_1,n_2,n_3)\in \mathbb{I}_k(\mathbf{eee}) \end{subarray}},\]
where $\mathbb{I}_k(\mathbf{eee})$ is defined in \eqref{eq2_5}.
The technique we have used can also be applied to the matrix $C_k^{(\mathbf{eee})}$.
In fact, one can show that
\[ \dim_\Q \langle \zeta^\m_\depth (n_1,n_2,n_3) \mid (n_1,n_2,n_3)\in \mathbb{I}_k(\mathbf{eee}) \rangle_\Q = \rank C_k^{(\mathbf{eee})}.\]
Our numerical experiment suggests that
\[ \sum_{k>0:{\rm even}} \rank C_k^{(\mathbf{eee})} x^k \stackrel{?}{=} \even(x)^3- \odd(x)^2-\even(x) \cusp(x) = \even(x)\odd(x)^2-\even(x)\cusp(x).\]
This dimension conjecture implies that the $\Q$-vector space $\Zsp_{k,3}^{al}$ spanned by elements $\zeta_\depth^\m(n_1) \zeta_\depth^\m(n_2,n_3) \ \big( (n_1,n_2,n_3)\in \mathbb{I}_k(\mathbf{eoo}) \big)$ is generated by $\zeta^\m_\depth(n_1,n_2,n_3)$  $ \big( (n_1,n_2,n_3)\in \mathbb{I}_k(\mathbf{eee}) \big)$, which was also pointed out by M. Hirose and N. Sato.

\appendix

\section{Rational realizations of the formal double zeta space}

In order to recall the proof of Proposition \ref{prop:tau_formula} from \cite{GKZ}, we briefly review works on rational realizations of the formal double zeta space.

Throughout this subsection, we assume $k$ to be even.
The formal double zeta space, denoted by $D_k$, is the $\Q$-vector space spanned by formal symbols $Z_{r,s},P_{r,s},Z_k$ \ ($r,s\ge1,r+s=k$) which are subject only to the relations
\begin{equation*}
\begin{aligned}
&Z_{r,s}+ Z_{s,r}+Z_{r+s}=P_{r,s} \qquad (r+s=k),\\
&\sum_{r+s=k} \left( \binom{s-1}{i-1} + \binom{s-1}{j-1} \right) Z_{r,s} =  P_{i,j} \qquad (i+j=k).
\end{aligned}
\end{equation*}
We note that our definition differs from the one in \cite[\S2]{GKZ}, because we use the opposite convention, i.e. their double zeta value $\zeta(r,s)$ equals our $\zeta(s,r)$.
An element in the space ${\rm Hom}_\Q( D_k,\Q)$, which will be our interest, is called the rational realization of the formal double zeta space.
Finding a rational realization is equivalent to solving the double shuffle equation in depth 2 
\begin{equation}\label{eq:ds_eq_dep2}
\begin{aligned}
&\mathfrak{Z}_k(x,y)+\mathfrak{Z}_k(y,x)+z_k \frac{x^{k-1}-y^{k-1}}{x-y} = \mathfrak{P}_k(x,y),\\
&\mathfrak{Z}_k(x,x+y)+\mathfrak{Z}_k(y,x+y) = \mathfrak{P}_k(x,y),
\end{aligned}
\end{equation}
with $\mathfrak{Z}_k(x,y)=\sum_{r+s=k} z_{r,s}x^{r-1}y^{s-1}\in \Q[x,y], \mathfrak{P}_k(x,y)=\sum_{r+s=k} p_{r,s}x^{r-1}y^{s-1}\in \Q[x,y]$ and $z_k\in \Q$.
In fact, the above solution induces an element $\varphi \in D_k^\vee$ given by
\[\varphi(Z_{r,s})=z_{r,s}, \ \varphi(Z_k)=z_k ,\ \varphi(P_{r,s})=p_{r,s}.\]
Solutions to the equations \eqref{eq:ds_eq_dep2} are also developed by Ecalle \cite{Ec} and Brown \cite{B2}, in which $p_{r,s}$ is always treated to be the product $z_rz_s$.

In \cite[Section 6]{GKZ}, an explicit element in ${\rm Hom}_\Q( D_k,\Q)$ is constructed by using extended odd period polynomial. 
Set $\widehat{V}_k=\bigoplus_{\substack{r+s=k\\r,s\ge0}}\Q x^{r-1}y^{s-1}$.
Let $\widehat{W}_k$ be the space of extended period polynomials:
\[ \widehat{W}_k=\{ P \in \widehat{V}_k \mid P\big| (1-T-T')=0\},\]
where $T=(\begin{smallmatrix}1&1\\0&1\end{smallmatrix}), T'=(\begin{smallmatrix} 1&0\\1&1 \end{smallmatrix})$.
The action of the group ${\rm PGL}_2(\Z)$ on the space $\Q(x,y)$ of rational functions is defined in the standard manner by $(P\big|\gamma)(x,y)=P(ax+by,cx+dy)$ for $\gamma=(\begin{smallmatrix}a&b\\c&d \end{smallmatrix})$ and extended to the group ring $\Z[{\rm PGL}_2(\Z)]$ by linearity.
The space $\widehat{W}_k$ splits into even (resp.\ odd) polynomial spaces $\widehat{W}_k^{+}$ (resp.\ $\widehat{W}_k^{-}$).
We remark that $\widehat{W}_k^{+}=W_k^+:=W_k^{+,0}\oplus \Q(x^{k-2}-y^{k-2}) $ the space of even period polynomials (no pole) and each element $P$ in $W_k^{+}$ (resp.\ $\widehat{W}_k^{-}$) has antisymmetric (resp.\ symmetric) property:
\begin{equation}\label{eq:ep_inv}
P\big| \varepsilon = -P \quad (\mbox{resp.}\ P\big| \varepsilon =P ),
\end{equation}
where $\varepsilon=(\begin{smallmatrix} 0&1\\1&0 \end{smallmatrix})$.
Note that $T'=\varepsilon T \varepsilon$.

\begin{proposition}{\cite[Proposition 5 (ii)]{GKZ}}\label{prop:poler_solution}
For $P(x,y)\in \widehat{W}_k^{-}$, let $Q:= \frac{1}{3} P\big|(T^{-1}+1)$. 
Then we have
\[ Q(x,y)+Q(y,x)=P(x,y),\qquad Q(x+y,y)+Q(x+y,x)=P(x,y).\]
\end{proposition}
\begin{proof}
Proof is the same with \cite[Proposition 5 (ii)]{GKZ}.
Note that $3Q(x,y)+3Q(y,x)=P\big|(T^{-1}+1)(1+\varepsilon)$.
Let $\delta=(\begin{smallmatrix} -1&0\\0&1 \end{smallmatrix})$.
It follows that $\delta T^{-1} \varepsilon = T' T^{-1}$.
By \eqref{eq:ep_inv} and $P\big|\delta=-P$, we see that
\[P\big|(T^{-1}+1)(1+\varepsilon)=P\big| (2+T^{-1}-\delta T^{-1} \varepsilon)= P\big| (3+(1-T-T')T^{-1}).\]
Then, the defining relation of $\widehat{W}_k$ shows $P\big|(T^{-1}+1)(1+\varepsilon)=3P$, which proves the first identity.
The second identity follows from $3Q(x+y,y)+3Q(x+y,x)=P\big|(T^{-1}+1)T(1+\varepsilon) =P\big|(2+T+T')=3P$.
\end{proof}

It is shown by Zagier \cite{Z4} that the space $M_k$ of modular forms is canonically isomorphic to $\widehat{W}_k^{-}\otimes \C$.
Thus, Proposition \ref{prop:poler_solution} says that each modular form provides a polar solution to the double shuffle equation \eqref{eq:ds_eq_dep2} with $z_k=0$.
We see that a solution corresponding to a cusp form does not have any pole and the solution obtained from the period polynomial of the Eisenstein series (already appeared in \eqref{eq:period_poly_Eis}!) does have a pole.
Removing poles, one may get a rational realization of the formal double zeta space.
The procedure to remove poles is not unique, but one way is given as follows.

\begin{proposition}{\cite[Supplement to Proposition 5]{GKZ}}\label{prop:supp_prop_5}
Take $\widehat{P}(x,y)=P(x,y) + \lambda(x^{k-1}y^{-1}+y^{k-1}x^{-1}) \in \widehat{W}_k^{-}$ to be $P(x,y)\in \Q[x,y]$.
Then we have a solution to \eqref{eq:ds_eq_dep2} with $\mathfrak{P}_k = P$, $z_k=-2\lambda$ and 
\[\mathfrak{Z}_k = \frac{1}{3} P\big|(T^{-1}+1)\varepsilon +\frac{\lambda}{6} \frac{x^{k-1}-y^{k-1}}{x-y} \bigg| (5- 3U +U\varepsilon)\varepsilon,\]
where $U=(\begin{smallmatrix} 1&-1\\1&0 \end{smallmatrix})$.
\end{proposition}

Proposition \ref{prop:tau_formula} is obtained by taking $\widehat{P}$ in Proposition \ref{prop:supp_prop_5} to be the period polynomial of the Eisenstein series \eqref{eq:period_poly_Eis}.
We now sketch the proof.
\begin{proof}[Proof of Proposition \ref{prop:tau_formula}]
Let 
\[ \lambda = - \frac{\beta_k}{2},\ P(x,y) =  \sum_{\substack{r+s=k\\r,s\ge1}} \beta_{r}\beta_{s} x^{r-1}y^{s-1} .\]
Then we have
\begin{align*}
\frac{1}{3} P\big|(T^{-1}+1)\varepsilon +\frac{\lambda}{6} \frac{x^{k-1}-y^{k-1}}{x-y} \bigg| (5- 3U +U\varepsilon)\varepsilon =\beta_k \sum_{\substack{r+s=k\\r,s\ge1}} \tau(r,s) x^{r-1}y^{s-1},
\end{align*}
where $\tau(r,s)$ is defined in Proposition \ref{prop:tau_formula}.
Since $\widehat{P}(x,y) := P(x,y) +\lambda \left(x^{k-1}y^{-1}+y^{k-1}x^{-1}\right) =   4 \widehat{G}_k(x,y) \in \widehat{W}_k^-$, by Proposition \ref{prop:supp_prop_5} we have
\begin{equation*}
\begin{aligned}
&\beta_k\tau(r,s)+\beta_k\tau(s,r) +\beta_k = \beta_r\beta_s \qquad (r+s=k), \\
& \sum_{r+s=k} \left( \binom{s-1}{i-1}+\binom{s-1}{j-1}\right) \beta_k\tau(r,s)= \beta_i\beta_j \qquad (i+j=k),
\end{aligned}
\end{equation*}
from which Proposition \ref{prop:tau_formula} follows.
\end{proof}

It is worth mentioning that the solution to the equations \eqref{eq:ds_eq_dep2} is not uniquely determined.
For example, Brown \cite[\S7]{B2} gave another solution (it is denoted by $\tau^{(2)}$).
These are different because of solutions to the linearized double shuffle equation in depth 2:
\begin{equation}\label{eq:lin_ds_dep2} 
\mathfrak{L}_k(x,y) + \mathfrak{L}_k(y,x)=\mathfrak{L}_k(x,x+y)+\mathfrak{L}_k(y,x+y)=0
\end{equation}
for $\mathfrak{L}_k(x,y) = \sum_{r+s=k} l_{r,s}x^{r-1}y^{s-1} \in \Q[x,y]$.
In fact, if $(\mathfrak{Z}_k,\mathfrak{P}_k,z_k)$ and $\mathfrak{L}_k$ are solutions to the equations \eqref{eq:ds_eq_dep2} and \eqref{eq:lin_ds_dep2}, respectively, then $(\mathfrak{Z}_k+\mathfrak{L}_k,\mathfrak{P}_k,z_k)$ is a solution to \eqref{eq:ds_eq_dep2}.
It is well-known that the $\Q$-vector space spanned by solutions to \eqref{eq:lin_ds_dep2} is generated by 
\[ \{x^{r-1},y^{s-1}\} = x^{r-1}\big((y-x)^{s-1}-y^{s-1}\big)+(y-x)^{r-1}\big( y^{s-1}-x^{s-1}\big) + y^{r-1}\big(x^{s-1}-(y-x)^{s-1}\big)\]
for $r,s\ge3 $ odd with $r+s=k$, where $\{,\}$ denotes the Ihara bracket (see \cite[Example 6.9]{B}).

We finally make a remark on an analogous result to Proposition \ref{prop:supp_prop_5}, whose proof is omitted.

\begin{proposition}
For $\widehat{P}(x,y)=P(x,y) + \lambda(x^{k-1}y^{-1}+y^{k-1}x^{-1}) \in \widehat{W}_k^{-}$, we have a solution to \eqref{eq:ds_eq_dep2} with $\mathfrak{P}_k = P$, $z_k=-2\lambda$ and 
\[\mathfrak{Z}_k = \frac{1}{3} P\big|(T^{-1}+1)\varepsilon - \lambda \left( \frac{y^{k-1}}{x} - \frac{x^{k-1}-y^{k-1}}{x-y} \right) + \frac{\lambda}{3k} \{ x^{-2},y^k\} \in \Q[x,y].\]
\end{proposition}

\section{Proof of \eqref{eq:dep2_pr2}}

We prove \eqref{eq:dep2_pr2}.
By \eqref{eq:formula_a}, the left-hand side of \eqref{eq:dep2_pr2} is reduced to
\begin{align*}
&\frac{1}{\beta_k} \sum_{\substack{i+j=k\\i,j:{\rm even}}} a_{i,j}(f) \beta_i\beta_j + \sum_{i+j=k} a_{i,j}(f)\\
&=\sum_{\substack{r+s=k\\r,s:{\rm odd}}} (-1)^{\frac{s-1}{2}} L_f^\ast(s) \frac{1}{\beta_k} \sum_{\substack{i+j=k\\i,j:{\rm even}}}\binom{i-1}{s-1} \beta_i\beta_j + \sum_{\substack{r+s=k\\r,s:{\rm odd}}} (-1)^{\frac{s-1}{2}} L_f^\ast(s) \sum_{i+j=k}\binom{i-1}{s-1} .
\end{align*}
Since $\sum_{i+j=k}\binom{i-1}{s-1}=\binom{k-1}{s}$ and $\beta_0=-\frac12$, we have
\begin{align*} 
&=\sum_{\substack{r+s=k\\r,s:{\rm odd}}} (-1)^{\frac{s-1}{2}} L_f^\ast(s)\left( \frac{1}{\beta_k} \sum_{j=2}^k \binom{j-1}{s-1} \beta_j\beta_{k-j}+ \frac{1}{2} \binom{k-1}{s-1}+\binom{k-1}{s} \right).
\end{align*}
Now use \eqref{eq:lambda_odd} to obtain
\begin{align*}
&= 3\sum_{\substack{r+s=k\\r,s:{\rm odd}}} (-1)^{\frac{s-1}{2}} \lambda(r,s) L_f^\ast(s)\\
&+\sum_{\substack{r+s=k\\r,s:{\rm odd}}} (-1)^{\frac{s-1}{2}} L_f^\ast(s)\left( \frac{1}{4} + \frac{1}{4} \binom{k-1}{s-1} - \frac14 \binom{k-1}{s}+\binom{k-1}{s}+\frac{1}{2} \binom{k-1}{s-1} \right) \\
&= 3\sum_{\substack{r+s=k\\r,s:{\rm odd}}} (-1)^{\frac{s-1}{2}} \lambda(r,s) L_f^\ast(s) + \frac14 \sum_{\substack{r+s=k\\r,s:{\rm odd}}} (-1)^{\frac{s-1}{2}} L_f^\ast(s) \left( 1+ 3\binom{k-1}{s-1} + 3\binom{k-1}{s}\right).
\end{align*}
Since $L_f^\ast(s)=(-1)^{\frac{k}{2}}L_f^\ast(k-s)$, we have 
\[ \sum_{\substack{r+s=k\\r,s:{\rm odd}}} (-1)^{\frac{s-1}{2}} L_f^\ast(s)=0 \ \mbox{and} \ \sum_{\substack{r+s=k\\r,s:{\rm odd}}} (-1)^{\frac{s-1}{2}} L_f^\ast(s) \left( \binom{k-1}{s-1} + \binom{k-1}{s}\right) =0.\]
As a result, we obtain
\[\frac{1}{\beta_k} \sum_{\substack{i+j=k\\i,j:{\rm even}}} a_{i,j}(f) \beta_i\beta_j + \sum_{i+j=k} a_{i,j}(f) = 3\sum_{\substack{r+s=k\\r,s:{\rm odd}}} (-1)^{\frac{s-1}{2}} \lambda(r,s) L_f^\ast(s),\]
and hence, the identity \eqref{eq:dep2_pr2} follows from \eqref{eq:KZ_relation}.


\end{document}